\documentclass[psamsfonts]{amsart}  % proc-l
\usepackage{amssymb}
\usepackage[dvips]{graphicx}

\usepackage{graphicx}
\usepackage{amssymb}                       
\usepackage{amsmath}
\usepackage{color}
\usepackage{times}

\usepackage[unicode,bookmarks,colorlinks]{hyperref}
\hypersetup{
   linkcolor=brickred,
}

\definecolor{mahogany}{cmyk}{0, 0.77, 0.87, 0}
\definecolor{salmon}{cmyk}{0, 0.53, 0.38, 0}
\definecolor{melon}{cmyk}{0, 0.46, 0.50, 0}
\definecolor{yellowgreen}{cmyk}{0.44, 0, 0.74, 0}
\definecolor{brickred}{cmyk}{0, 0.89, 0.94, 0.28}
\definecolor{OliveGreen}{cmyk}{0.64, 0, 0.95, 0.40}
\definecolor{RawSienna}{cmyk}{0, 0.72, 1.0, 0.45}
\definecolor{ZurichRed}{rgb}{1, 0, 0} % Red of svgnames

\usepackage{fancyhdr}
\usepackage{amsmath,amstext,amssymb,amsopn,amsthm}
\usepackage{amsmath,amssymb,amsthm}
\usepackage[mathscr]{eucal}

\date{\today}

\begin{document}

%\date{}

%\newtheorem{thm}{Theorem}
%\numberwithin{thm}{Theorem}
\newtheorem{lemma}[thm]{Lemma}
%\newtheorem{corr}[thm]{Corollary}[sections]
%\numberwithin{corollary}[thm]{section}
\newtheorem{proposition}{Proposition}
\newtheorem{theorem}{Theorem}[section]
\newtheorem{deff}[thm]{Definition}
\newtheorem{case}[thm]{Case}
%\numberwithin{deff}{section}
\newtheorem{prop}[thm]{Proposition}
%\numberwithin{equation}{subsection}
%\numberwithin{equation}{section}
\newtheorem{example}{Example}

\newtheorem{corollary}{Corollary}

\theoremstyle{definition}
\newtheorem{remark}{Remark}

\numberwithin{equation}{section}
\numberwithin{definition}{section}
%\numberwithin{problem}{section}
\numberwithin{corollary}{section}
%\numberwithin{proposition}{subsection}

\numberwithin{theorem}{section}

\numberwithin{remark}{section}
\numberwithin{example}{section}
\numberwithin{proposition}{section}

\newcommand{\gap}{\lambda_{2,D}^V-\lambda_{1,D}^V}
\newcommand{\gapR}{\lambda_{2,R}-\lambda_{1,R}}
\newcommand{\bD}{\mathrm{I\! D\!}}
\newcommand{\calD}{\mathcal{D}}
\newcommand{\calA}{\mathcal{A}}

\newcommand{\conjugate}[1]{\overline{#1}}
\newcommand{\abs}[1]{\left| #1 \right|}
\newcommand{\cl}[1]{\overline{#1}}
\newcommand{\expr}[1]{\left( #1 \right)}
\newcommand{\set}[1]{\left\{ #1 \right\}}

\newcommand{\cadlag}{c\`adl\`ag}
\newcommand{\calC}{\mathcal{C}}
\newcommand{\calE}{\mathcal{E}}
\newcommand{\calF}{\mathcal{F}}
\newcommand{\Rd}{\mathbb{R}^d}
\newcommand{\BR}{\mathcal{B}(\Rd)}
\newcommand{\R}{\mathbb{R}}
\newcommand{\al}{\alpha}
\newcommand{\RR}[1]{\mathbb{#1}}
\newcommand{\ga}{\gamma}
\newcommand{\om}{\omega}
\newcommand{\A}{\mathbb{A}}
\newcommand{\bH}{\mathbb{H}}

\newcommand{\bb}[1]{\mathbb{#1}}
\newcommand{\bI}{\bb{I}}
\newcommand{\bN}{\bb{N}}

\newcommand{\uS}{\mathbb{S}}
\newcommand{\M}{{\mathcal{M}}}
\newcommand{\calB}{{\mathcal{B}}}

\newcommand{\W}{{\mathcal{W}}}

\newcommand{\m}{{\mathcal{m}}}

\newcommand {\mac}[1] { \mathbb{#1} }

\newcommand{\bC}{\Bbb C}

\newtheorem{rem}[theorem]{Remark}
\newtheorem{dfn}[theorem]{Definition}
\theoremstyle{definition}
\newtheorem{ex}[theorem]{Example}
\numberwithin{equation}{section}

\newcommand{\Pro}{\mathbb{P}}
\newcommand\F{\mathcal{F}}
\newcommand\E{\mathbb{E}}
\newcommand\e{\varepsilon}
\def\H{\mathcal{H}}
\def\t{\tau}

\newcommand{\blankbox}[2]{%
  \parbox{\columnwidth}{\centering
%    Set fboxsep to 0 so that the actual size of the box will match the
%    given measurements more closely.
    \setlength{\fboxsep}{0pt}%
    \fbox{\raisebox{0pt}[#2]{\hspace{#1}}}%
  }%
}
%\copyrightinfo{2007}{American Mathematical Society}
%\begin{document}

%\copyrightinfo{2007}{American Mathematical Society}
%\begin{document}
%\maketitle

\title[Riesz transforms]
 {Sharp martingale inequalities and applications to Riesz transforms on manifolds, Lie groups and Gauss  space}

\author{Rodrigo Ba\~nuelos}\thanks{R. Ba\~nuelos is supported in part  by NSF Grant
\# 0603701-DMS}
\address{Department of Mathematics, Purdue University, West Lafayette, IN 47907, USA}
\email{banuelos@math.purdue.edu}
\author{Adam Os\c ekowski}
\address{Department of Mathematics, Informatics and Mechanics, University of Warsaw, Banacha 2, 02-097 Warsaw, Poland}
\email{ados@mimuw.edu.pl}
\thanks{A. Os\c ekowski is supported in part by Polish Ministry of Science and Higher Education (MNiSW) grant IP2011 039571 `Iuventus Plus'}

\dedicatory{In memory of Don Burkholder}



\begin{abstract} 
We prove  new sharp  $L^p$, logarithmic, and weak-type inequalities for martingales under the assumption of differentially subordination.   The $L^p$ estimates are ``Fyenman-Kac" type versions of Burkholder's celebrated martingale transform inequalities.  From the martingale  $L^p$ inequalities  we obtain that Riesz transforms on manifolds of nonnegative Bakry-Emery Ricci curvature have exactly the same $L^p$ bounds as those known for Riesz transforms in the flat case of $\R^n$.   From the martingale logarithmic and weak-type inequalities  we obtain similar inequalities for Riesz transforms on compact Lie groups and spheres. Combining the estimates for spheres with Poincar\'e's  limiting argument, we deduce the corresponding results for Riesz transforms associated with the Ornstein-Uhlenbeck semigroup,  thus providing  some extensions of P.A. Meyer's $L^p$ inequalities. 

\end{abstract}

\maketitle

\tableofcontents

\section{Introduction}
As evidenced in \cite{AppBan}, \cite{Ar}, \cite{BanBau}, \cite{BB}, \cite{BanJan}, \cite{BMH}, \cite{BO}, \cite{BW}, \cite{GMS}, 
\cite{O1}, \cite{O2} and many other papers, martingale inequalities play a fundamental role in obtaining sharp  behavior of $L^p$ bounds for numerous important singular integrals and Fourier multipliers operators.  Such operators include the classical first and second order Riesz transforms and a large class of multipliers obtained from certain transformation of the L\'evy-Khintchine formula, see \cite{BB}.  There has also been considerable interest in finding the exact values of various norms of other closely related operators, most notably the Beurling-Ahlfors transform on the complex plane $\bC$ and on $\R^n$ where the martingale techniques have been extremely useful. For an overview of many of these problems and their applications, we refer the reader to \cite{Ban}.
One of the motivations for investigating sharp estimates for such operators comes from the papers of Donaldson and Sullivan \cite{DS}, and Iwaniec and Martin \cite{IM1}, \cite{IM2}, in which it was pointed out that good estimates for the $L^p$ norm of the Riesz transforms on $\R^n$ and the Beurling-Ahlfors operator on $\mathbb{C}$ have important consequences in the study of quasiconformal mappings, related nonlinear geometric PDEs as well as in the $L^p$-Hodge decomposition theory.  For more on this connections, see also \cite{AstIwaMar}, \cite{Iwa}, \cite{IM1}. The purpose of this paper is to continue this line of research and to investigate explicit (tight) $L^p$,  weak-type and logarithmic inequalities for Riesz transforms on manifolds of nonnegative Ricci curvature, Lie groups and Gauss space.  
When restricted to the torus and $\R^n$, several of these bounds are sharp for Riesz transforms and hence they cannot be improved in general.  

We start with some necessary notation and present a brief review of related results from the literature. Suppose that $M$ is a complete Riemannian manifold equipped with the corresponding gradient $\nabla_M$ and the Laplace-Beltrami operator $\Delta_M$. Then $-\Delta_M$ is positive and the Riesz transform
\begin{equation}\label{generalriesz}
R^M=\nabla_M\circ (-\Delta_M)^{-1/2}
\end{equation}
is a well-defined operator on $L^2(M)$ (actually, an isometry). From this the interesting question of whether $R^M$ extends to a bounded operator on $L^p(M)$ for other $p$'s immediately arises. The first results in this direction are those of Riesz \cite{R} and concern the cases $M=\R$ and $M=\mathbb{S}^1$ where the operators reduce to the Hilbert transform.  Riesz proved that the Hilbert transform   
can be extended to a bounded operator on $L^p$, for $1<p<\infty$, but not for $p=1$ or $p=\infty$. This result was generalized by Calder\'on and Zygmund \cite{CZ} to Riesz transforms on  $\R^n$. That is, these operators also extend to bounded operators on $L^p$ if and only if $1<p<\infty$.  
 These are the classical results that the reader can find in Stein \cite{St}.  

The question concerning the precise value of the $L^p$ norms of the Hilbert  transform $R^{\R}$ and $R^{\mathbb{S}^1}$ was answered by Pichorides in \cite{P}, where it is proved that \
\begin{equation}\label{IwanMarPich1}
||R^{\R}||_{L^p(\R)\to L^p(\R)}=||R^{\mathbb{S}^1}||_{L^p(\mathbb{S}^1)\to L^p(\mathbb{S}^1)}=\cot\left(\frac{\pi}{2p^*}\right),\qquad 1<p<\infty,
\end{equation}
where 
$$p^*=\max\{p,p/(p-1)\}.
$$ 
With this we can write 
\begin{equation}\label{p^*}
p^*-1=\begin{cases} \frac{1}{p-1},  \hskip4mm  1<p\leq 2,\\
p-1 , \hskip3mm  2\leq p <\infty, 
\end{cases}
\end{equation}
which is the constant appearing in Burkholder's \cite{Bu1} celebrated work on inequalities for martingale transform.  This   quantity  will appear many times in this paper.  The bound given by  \eqref{IwanMarPich1} has been considerably extended by Iwaniec and Martin \cite{IM2} and Ba\~nuelos and Wang \cite{BW}. For a given $n$, introduce the directional Riesz transforms $R_1, R_2, \ldots, R_n$ defined by
$$
R_j=\partial_j\circ(-\Delta_{\R^n})^{-1/2},\,\,\,\, j=1,\,2,\,\ldots,\,n, 
$$
and note that
\begin{equation}
 R^{\R^n}=(R_1,R_2,\ldots,R_n)
\end{equation}
It turns out that the $L^p$ norms of the transforms $R_j$ do not depend on the dimension and are equal to Pichorides' constants.  That is, 
\begin{equation}\label{Pic}
||R_j||_{L^p(\R^n)\to L^p(\R^n)}=\cot\left(\frac{\pi}{2p^*}\right),\qquad 1<p<\infty,
\end{equation}
for all $\,j=1,\,2,\,\ldots,\,n.$ This was proved in \cite{IM2} with the use of the so-called method of rotations. The paper \cite{BW} develops a completely different proof which rests on martingale methods and which has a lot flexibility in its range of applications. 

The papers \cite{IM2} and \cite{BW} also contain tight information on the $L^p$ norm of the vectorial Riesz transform $R^{\R^n}$. Iwaniec and Martin proved that
\begin{equation}\label{IwanMarPich2}
 ||R^{\R^n}||_{L^p(\R^n)\to L^p(\R^n)}\leq 2\sqrt{2}\cot\left(\frac{\pi}{2p^*}\right),\qquad 2\leq p<\infty,
 \end{equation}
while Ba\~nuelos and Wang showed that
\begin{equation}\label{InBW}
 ||R^{\R^n}||_{L^p(\R^n)\to L^p(\R^n)}\leq 2(p^*-1),\qquad 1< p<\infty.
\end{equation}
For large $p$, the latter bound is slightly worse than the former; on the other hand, \eqref{InBW} works in the full range $1<p<\infty$. However, while we know the bound for the directional Riesz transforms in \eqref{Pic} is sharp, the sharp bound for $||R^{\R^n}||_{L^p(\R^n)\to L^p(\R^n)}$ remains open.  This is stated in \cite{Ban} as {\it Problem 6} where it is also conjectured   that the sharp bound should be  $\cot\left(\frac{\pi}{2p^*}\right)$.  We note that both \eqref{IwanMarPich2} and \eqref{InBW} do not give the sharp bound even when $p=2$, which by the Fourier transform is $1$.

One may study similar statements for Riesz transforms on manifolds as defined by \eqref{generalriesz} or Riesz transforms associated with the Ornstein-Uhlenbeck semigroup. Since Stein \cite{St0} introduced the Riesz transforms on compact Lie groups and applied Littlewood-Paley inequalities to prove the $L^p$-boundedness of these operators, many mathematicians have investigated the properties of Riesz transforms on various geometric settings. In analogy with the case $M=\R^n$, Strichartz \cite{Str} raised the question concerning the structure of the manifold $M$ which guarantees that $R^M$ extends to the bounded operator on $L^p(M)$ for $1<p<\infty$. The further crucial issue is, for such $M$, to identify the exact value of $||R^M||_{L^p(M)\to L^p(M)}$, or at least, provide a good upper bound for it. The literature on Riesz transforms on manifolds and Lie groups  is quite large by now and it would be impossible for us to give complete references here.  We refer the interested reader to Arcozzi \cite{Ar}, Auscher and Coulhon \cite{Ae0}, Auscher et al. \cite{Ae}, Bakry \cite{Bak}, Baudoin and Garofalo \cite{BauGar}, Carbonaro and Dragi\v cevi\'c \cite{CarDra}, Coulhon and Duong \cite{CT}, Coulhon and Dungey \cite{CT2}, J.-Y. Li \cite{Li0}, X.-D. Li \cite{Li}, Lohou\'e \cite{Lo} and  Strichartz \cite{Str} where many bounds are provided under curvature and other geometric assumptions on $M$.  These papers also contain many references to the enormous literature on Riesz transforms. 

In \cite{BanBau}, Ba\~nuelos and Baudoin studied a class of operators obtained by projections (conditional expectations) of certain martingales transforms on manifolds under very general conditions.  These operators contain the second order Riesz transforms on $\R^n$.  
Let 
 ${M}$ a be smooth manifold  with a volume measure $\mu$ and consider the second order operator  
\begin{equation}\label{LV}
L=-\frac{1}{2} \sum_{i=1}^n X_i^* X_i+V, 
\end{equation}
where  
$X_1,\cdots, X_n$ are locally Lipschitz vector fields defined on ${M}$, 
$X_i^*$ is the formal adjoint of $X_i$ with respect to $\mu$ and  $V:{M} \rightarrow \mathbb{R}$ is a non-positive smooth  {\it potential}. Denote by  $P_t$ the heat semigroup of the operator $L$ and let $A_{ij} : [0,+\infty) \times {M} \rightarrow \mathbb{R}$, $1\le i,j \le n$ be bounded  smooth real valued functions. Next, consider the $n\times n$ matrix  $A(t,x)=\left(A_{ij}\right)$ and set 
$$
\|A\|=\| |A(t,x)|\|_{L^{\infty} ([0,+\infty) \times \mathbb{M})},
$$
 where $|A(t,x)|$ is the usual quadratic norm of the $n\times n$ matrix $A(t,x)$. We assume that $\|A\|<\infty$ and put 
\begin{equation}
\mathcal{S}_A f=\sum_{i,j=1}^n\int_{0}^{\infty}   P_t  X_i^*  A_{ij} (t, \cdot) X_j P_t f\mbox{d}t. 
\end{equation}
It is then proved in \cite{BanBau} that there exists a constant $C_p$ depending only on $p$ such that 
\begin{equation}\label{Schro1}
\| \mathcal{S}_A f \|_{L^p(M)} \le \|A\| C_p \| f \|_{L^p(M)}, \,\,\,\, 1<p<\infty
\end{equation} 
and that if $V=0$ we can take $C_p=(p^*-1)$.  For the case $V=0$ one may apply the celebrated martingale transform inequalities of Burkholder \cite{Bu1}.  However, in order  to obtain inequality \eqref{Schro1} for non-zero $V$, a novel martingale inequality is  needed which provides an extension of the classical Burkholder-Davis-Gundy inequalities for what one may call ``Schr\"odinger-type" martingale transforms.   The new inequalities are Theorems 2.5 and 2.6 in \cite{BanBau}.  In this paper we prove sharp versions of these results.  The new inequalities are contained in  Theorems \ref{thmmart},  \ref{potential-nonmax} and \ref{potential-Thm}.  The arguments in \cite{BanBau} and  our new sharp martingale inequality  \eqref{mart_inequality} give 

\begin{theorem}\label{manifolds2}  If $L$ is as defined in \eqref{LV} with $V$ non-positive, then for $1<p<\infty$,  we have 
\begin{equation}\label{Schro2}
\| \mathcal{S}_A f \|_{L^p(M)} \le \|A \|(p^*-1)  \| f \|_{L^p(M)}.
\end{equation}
 \end{theorem}
As already mentioned, the operators $\mathcal{S}_A$  include the second order Riesz transforms on $\R^n$ (see \cite{BanBau} for details) and hence given the results in \cite{BO}, the estimate \eqref{Schro2} cannot be improved, in general.  The novelty here again is that the behavior of the constant is the same as in the case when the potential is identically zero and the manifold is $\R^n$.   The bound $ \|A \|(p^*-1)$   should be compared with the bound $8  \| A \| (p^*-1)  \frac{p^4}{(p-1)^2}$ given in \cite[Corollary 3.2]{BanBau} which is $O(p^{3})$, as $p\to\infty$ and $O(\frac{1}{p-1})^{4}$, as $p\to 1$.  It is also interesting to note here that this theorem is proved with no geometric assumptions on the manifolds which is rare with  these type of results.

In the  papers \cite{Li, Li2, Li3}, Li  extends the Gundy-Varopoulos \cite{GV} probabilistic representation of Riesz transforms on $\R^n$ and its variant for the Beurling-Ahlfors operator by Ba\~nuelos-M\'endez \cite{BMH}, to manifolds under curvature assumptions   and obtains explicit $L^p$ bounds  which in some cases are similar to those  for the classical Riesz transforms on $\R^n$ given in \eqref{InBW}.  For example,  in  \cite{Li} (see Theorem 1.4 and Corollary 1.5) it is shown that the Riesz transforms on manifolds of nonnegative Ricci curvature are bounded on $L^p$ with bounds not exceeding $2(p^*-1)$.  However, as noted in \cite[Remark 2.1]{BanBau}, Li's paper \cite{Li}  contains a gap.  Similar gap exists in \cite{Li3} where Riesz transforms on differential forms are studied and applications to a Beurling-Ahlfors type operator on manifolds are given.  This  gap, which occurs in the probabilistic representation of the Riesz transforms and the Beurling-Ahlfors operator, is not fatal.  Indeed, as observed in \cite{BanBau}, the correction simply requires removing a non-adaptive term from inside a stochastic integral to outside the stochastic integral.  Unfortunately, and this is where the serious part of the gap arises, once this change  is made unless the curvature is identically zero,  the classical Burkholder-Davis-Gundy inequalities cannot be applied nor can one apply the  sharp martingale inequalities of Burkholder which are used in the flat case of $\R^n$ to obtain the $2(p^*-1)$ bound in \cite{BW}, and similar bounds for the Beurling-Ahlfors operator in \cite{BMH}.  For this reason, a new martingale inequality  is required.  This new martingale inequality,  which fixes the gap and restores Li's results (but not with his claimed constants), was proved  in Ba\~nuelos and Baudoin \cite{BanBau}.  Subsequently,  Li \cite{Li4},  \cite{Li5} elaborates  further on the corrections in \cite{BanBau} and, by substituting the explicit  constants given in \cite[Theorem 2.6]{BanBau} and  those of his Proposition 6.2 in \cite{Li2} gives explicit bounds which although not the same as those originally claimed are  $O(p^*-1)^{3/2}$, as $p\to 1$ and $p\to \infty$.

The new sharp martingale inequality in this paper, \eqref{mart_inequality} of Theorem \ref{thmmart} below, can be used to restore Li's bounds as originally claimed.   Rather than giving  a complete list of all the results we can prove with the new inequalities, we only give a couple of concrete  examples.  The following is a result claimed in  Theorem 1.4 and Corollary 1.5 in \cite{Li}.

\begin{theorem}\label{manifolds1}  Let $(M, g)$ be a complete Riemannian manifold with a Riemannian
metric $g$.  For $\phi\in C^2(M)$, set $L=\Delta -\nabla\phi\cdot\nabla$  and  $d\mu=e^{-\phi(x)}\sqrt{det(g(x)}dx$.  Let $Ric(L)=Ric+\nabla^2\phi$, where $\nabla^2\phi$ is the Hessian of $\phi$,  denote the Bakry-Emery Ricci curvature of $L$.  Set $R_0^L=\nabla\circ{(-L)^{-1/2}}$ and assume $Ric{(L)}\geq 0$. Then for all $f\in C_{0}^{\infty}(M)$, 
\begin{equation}\label{InLi1}
\|R_0^L(f)\|_{L^p(M)}\leq 2(p^*-1)\|f\|_{L^p(M)}, \,\,\, 1<p<\infty. 
\end{equation}
In particular, if $M$ is a complete Riemannian manifold of non-negative Ricci
curvature and we consider the Riesz transforms $R^M=\nabla_M\circ (-\Delta_M)^{-1/2}$ as defined in \eqref{generalriesz}, then 
\begin{equation}\label{InLi3}
\|R^M(f)\|_{L^p(M)}\leq 2(p^*-1)\|f\|_{L^p(M)}, \,\,\, 1<p<\infty. 
\end{equation}
Furthermore, if we set $R_a^L=\nabla\circ{(a-L)^{-1/2}}$, for $a>0$,  then under the assumption that 
$Ric{(L)}\geq -a$,  
\begin{equation}\label{InLi2}
\|R_a^L(f)\|_{L^p(\mu)}\leq 2(1+4\|\tau\|_p)(p^*-1)\|f\|_{L^p(\mu)}, \,\,\, 1<p<\infty,  
\end{equation}
where $\tau$ is the first exit time of the $3$-dimensional Brownian motion from the unit ball in $\R^3$ starting from $0$. 
\end{theorem}

In \cite{CarDra}, Carbonaro and Dragi\v cevi\'c  used  Bellman function technique to prove that for any $a\geq 0$, 
\begin{equation}\label{CarDra1}
\|R_a^L(f)\|_{L^p(\mu)}\leq 12(p^*-1)\|f\|_{L^p(\mu)}, \,\,\, 1<p<\infty.   
\end{equation}
The Bellman function  techniques were  applied to study bounds for second order Riesz transforms on $\R^n$ in \cite{NazVol}.  For other similar applications, see  \cite{DraVol0}, \cite{DraVol1},  \cite{DraVol2}.  
Since (as pointed out in \cite{CarDra}) it is well known that $\|\tau\|_p\sim p$ as $p\to\infty$,  the constant in  \eqref{InLi2} is of order $p^2$, as $p\to\infty$, and thus the Carbonaro--Dragi\v cevi\'c bound \eqref{CarDra1} is better than the bound given by  \eqref{InLi2}.  Here we can improve on the estimate \eqref{InLi2} to obtain a bound valid for all $a>0$ which, although not as good as the one for $a=0$ in \eqref{InLi1}, it is of the form $c(p^*-1)$, with $c<8$,  improving on \eqref{CarDra1}.   Indeed, using  Theorem \ref{potential-nonmax} and Proposition 6.2 in \cite{Li2} we obtained (as in the proof of Theorem 2.4 (iii) in \cite{Li5})  
\begin{equation}\label{InLi44}
\|R_a^L(f)\|_{L^p(\mu)}\leq 2D_p\frac{p}{\sqrt{2(p-2)}}\|f\|_{L^p(\mu)}, \,\,\, 3\leq p<\infty,  
\end{equation}
where $D_p$ is Davis's constant in \eqref{Davis}. 
Using the fact that $D_p\leq 2\sqrt{p}$ (see Remark \ref{remarkdavis}) we see that for $3\leq p<\infty$, $$2D_p\frac{p}{\sqrt{2(p-2)}}\leq 2\sqrt{2}\,p\sqrt{\frac{p}{p-2}}\leq 2\sqrt{6}p\leq  3\sqrt{6}(p-1).$$ 
 These  calculations give that for any $a>0$, 
\begin{equation}\label{InLi4}
\|R_a^L(f)\|_{L^p(\mu)}\leq \begin{cases}
\displaystyle 2(1+4\|\tau\|_p)\,(p^*-1)\|f\|_{L^p(\mu)},& \mbox{if }1<p<3,\\ 
\displaystyle 3\sqrt{6}(p^*-1)\|f\|_{L^p(\mu)},& \mbox{if }3\leq p<\infty.
\end{cases}
\end{equation}
Since $\|\tau\|_p\leq \|\tau\|_3$ for $1<p<3$, we can also replace the first term by an absolute constant. How big is $\|\tau\|_3$? This can be easily estimated given that we know $\E_{0}\tau=\frac{1}{3}$.  Indeed, it follows from the strong Markov  property (see \cite[p. 316]{banold}) that for all $\alpha>0$, 
$$
\int_{\alpha}^{\infty}\Pro_{0}\{\tau>t\}dt\leq \E_{0}(\tau)\Pro_{0}\{\tau>\alpha\}.
$$   
Now, for a fixed $k>1$, we multiply both sides by $k\alpha^{k-1}$ and integrate over $[0,\infty)$ with respect to $\alpha$, obtaining $\E_{0}\tau^{k+1}\leq (k+1)\E_{0}(\tau)\E_{0}\tau^{k}$.  Iterating this we find that for any $k=1, 2, \dots$,  $\E_{0}\tau^{k}\leq k! \left(\E_0{\tau}\right)^k$.  In particular, $\|\tau\|_3\leq \left(\frac{2}{9}\right)^{1/3}$  and therefore  \eqref{InLi4} yields
\begin{equation}\label{InLi5}
\|R_a^L(f)\|_{L^p(\mu)}\leq \begin{cases}
\displaystyle 2\left[1+4\left(\frac{2}{9}\right)^{1/3}\right]\,(p^*-1)\|f\|_{L^p(\mu)},& \mbox{if }1<p<3,\\ 
\displaystyle 3\sqrt{6}(p^*-1)\|f\|_{L^p(\mu)},& \mbox{if }3\leq p<\infty. 
\end{cases}
\end{equation}

Of course, we picked the cutoff value $3$ for no particular reason other than the fact that it is larger than 2 (required for the bound in \eqref{InLi44}) and that both estimates in \eqref{InLi4} give less than $12$.  What is clear is that the higher we go with this split, the better the bound in the second term and the worse the bound in the first term.  Perhaps more interesting is to note that asymptotically, as $p\to\infty$, we get the behavior $2\sqrt{2}p$ from \eqref{InLi4} for all $a's$, while for $a=0$ we have $2p$ from \eqref{InLi1}. On the other hand, as $p\to 1$ we get behavior $\frac{14}{3(p-1)}$ from \eqref{InLi2}.

We note here that Theorem \ref{manifolds1} includes the classical case of the Riesz transforms for the Ornstein-Uhlenbeck (Gauss space) semigroup on $\R^n$.  In this case, as already mentioned,  the bound was established by Arcozzi \cite{Ar} and it is, asymptotically in $p$, as $p\to 1$ and $p\to \infty$, best possible; see \cite{LarCoh}.   

With the bounds of Theorem \ref{manifolds1} one can also ``restore"  {\it Conjecture 1} made in \cite{Li} that under the assumption of $Ric{L}\geq 0$, the $L^p$ norm of the operator $R_0$, for $1<p<\infty$,  should be bounded below by $c(p^*-1)(1+o(1))$, for some universal constant $c$.

The Beurling-Ahlfors operator on $\R^n$ acting on $k$-forms is defined by $S_k=(d^*\,d-d\,d^*){\square_k}^{-1}$ where $\square_k$ is the Hodge Laplacian acting on $k$ forms, $d$ is the exterior differential operator and $d^*$ is its adjoint. 
The operator  $S_k$  was studied in \cite{DS} in connections to ``Quasiconformal 4-manifolds" and properties  of its $L^p$ norm on $\R^n$ were  investigated in \cite{IM2}. In particular, with $\|S\|_p=\max_{0\leq k\leq n}\|S_k\|_p$, where $\|S_k\|_p$ is the $L^p$ norm of $S_k$,  it is proved in \cite{IM2} that 
\begin{equation}\label{IMestimate}
(p^*-1)\leq \|S\|_{p}\leq c(n+1)\,p^2,\,\,1<p<\infty,
\end{equation}
where  $c$ is a universal constant independent of $n$.  In \cite{IM2},  the authors also make the far reaching conjecture that 
%\begin{conjecture}\label{IwaMar1}
for all $n\geq 2$,
$\|S\|_{p}=p^*-1, \,\,1<p<\infty.$  The lower bound follows from Lehto \cite{Le}.
 %\end{conjecture}
 The conjectured upper bound remains open even in the case $n=2$ where it is  known as the {\it Iwaniec Conjecture} \cite{Iw}.  The best known upper bound when $n=2$ is $1.575(p^*-1)$, valid for all $1<p<\infty$ (see Ba\~nuelos and Janakiraman \cite{BanJan}), and $1.4(p^*-1)$,  valid for all $p\geq 1000$ (see Borichev, Janakiraman and Volberg \cite{BorJanVol}).  It is well known that this conjecture has  many connections to problems in quasiconformal mappings as well as being related, via the Burkholder function (\eqref{Up} below),  to the celebrated question of Morrey on rank-one-convex and quasi-convex functions.  For these connections, see \cite{AstIwaMar}, \cite{AstIwaPraSak}, \cite{Ban} \cite{Iwa}.   
 
In Ba\~nuelos and Lindeman \cite{BanLin} a  
representation of operator $S_k$ on $\R^n$, for any $n\geq 2$,  is given in terms of martingale transforms and  from this the estimate in \eqref{IMestimate} is improved to  
$$
\|S\|_{p}\leq \begin{cases} \left({n}+2\right)(p^*-1),
&\text{$2\leq n\leq 14,\text{ and even}$}\\ 
{(n+1)}(p^*-1),&\text{$3\leq n\leq 13,\text{ and odd}$}\\
{\left({4n\over 3}-2\right)(p^*-1)},&\text{otherwise}.
\end{cases}
$$
Using  the martingale techniques from \cite{BMH},  Hyt\"onen  \cite{Hyt} improved this to 
$$\|S\|_{p}\leq \left(\frac{n}{2}+1\right) (p^*-1), \,\,\,\, 1<p<\infty, 
$$ for all $n\geq 2$.  This is, as of now, the best known bound on $\R^n$ valid for all $n$.  Other improvements on the results in \cite{BanLin} are contained in Petermichl, Slavin, and Wick in \cite{PetSlaBre1}.  The weaker problem of proving that the norm $\|S\|_p$  is bounded above with a constant independent of the dimension $n$ (even at the expense of giving the right dependence on $p$) remains and interesting open problem;   see \cite[Problem 10]{Ban}.

Returning to the setting of manifolds,  Li \cite{Li3} extends the probabilistic formula in \cite{BanLin} and \cite{BMH} to give a probabilistic representation for $S_k$ on stochastically complete Riemannian manifolds with Weitzenb\"oak  curvature bounded below.  From this and martingale inequalities he concludes that  there exists a constant depending  on $k$ such that 
\begin{equation}\label{li1}
\|S_k\|_p\leq C_k(p^*-1)^{3/2}, \,\,\,\, 1<p<\infty,
\end{equation} 
when the curvature is bounded below by zero and that when the curvature is zero, then 
\begin{equation}\label{li2}
\|S_k\|_p\leq C_k(p^*-1), \,\,\,\, 1<p<\infty. 
\end{equation}
 Unfortunately, the error in the representation formula for functions in \cite{Li} is repeated in the representation formula on differential forms in  \cite{Li3}.  As before, the correction is trivially achieved by moving the non-adaptive term to outside the stochastic integral.  But also as before, once this is done the classical martingale inequalities cannot be applied.  As observed by Ba\~nuelos and Baudoin in \cite[Remark 2.1]{BanBau}, Theorem 2.6 in \cite{BanBau} restores Li's original results up to universal constants depending only on $p$.  Following \cite{BanBau}, Li \cite{Li4} and \cite{Li5} elaborates  further on these corrections and again substituting the explicit constants obtained in \cite[Theorem 2.6]{BanBau} and his Proposition 6.2 in \cite{Li2}, restores the above bounds.

 As before, using Theorem \ref{thmmart} below, we obtain improvements of Li's results.  Once again, rather than listing  all the results explicitly, we give an example.
   \begin{theorem}\label{Li-forms} 
 Let  M be a complete and stochastically complete Riemannian
manifold of nonnegative Weitzenb\"oak curvature.  Then 
 \begin{equation}\label{Li-forms1}
 \|S_k\|_p\leq C_k(p^*-1), \,\,\,\, 1<p<\infty,
 \end{equation}
 where $C_k$ is a constant depending on $k$.  
 \end{theorem}

 If we assume that the Weitzenb\"oak  curvature is bounded below by $-a$ for some nonnegative constant $a$, then our inequalities can be used to obtain estimates on the operator $S_k=(d^*\,d-d\,d^*){(a+\square_k)}^{-1}$ as well as Riesz transforms on forms.   We leave these to the interested reader referring to \cite{Li4} and \cite{Li5}. 
   
 Finally, while the results in \cite{CarDra} show the effectiveness of the Bellman function techniques to study the boundedness of  the Riesz transform on manifolds under the Bakry-Emery curvature assumptions, those techniques have not been applied (to the best of our knowledge) to study the Riesz transforms, or the the Beurling-Ahlfors operator, on differential forms on manifolds under  the Weitzenb\"oak  curvature  assumptions.  We believe such approach could produce interesting surprises. 
 
We now turn our attention to weak-type and logarithmic inequalities.  The problem of studying the asymptotic behavior of the $L^p$ bounds of Riesz transforms on manifolds as $p\to 1$ and $p\to\infty$ is attributed to Le Jan;  see \cite[Problem 1]{Li}. On $\R^n$,  the interest in the asymptotic behavior of these constants has a long history, going back to Marcinkiewicz, Zygmund and many others. For example,  see  \cite[Chapter XII]{Zyg}, where it is shown that for sublinear operators with with $L^p$ bounds of the form  $(p^*-1)$ as $p\to 1$ and $p\to\infty$,  one can obtain exponential and $LLogL$ inequalities.  This behavior also points to weak-type $(1, 1)$ inequalities and to $H^1$ and $BMO$ bounds.  The $H^1$ and $BMO$ topics are not explored in this paper.   We do point out, however, that to the best of our knowledge, weak-type $(1,1)$ inequalities for Riesz transforms on general manifolds of nonnegative Ricci curvature are not known. We  believe such inequalities should hold.  In the same way, there are currently no weak-type $(1,1)$ inequalities for the Riesz transforms on Gauss space which hold in infinite dimension. We refer the reader to \cite{Ban}, Remark 3.4.2 and Problem 8, for more information about the problem of weak-type $(1,1)$ behavior for Riesz transforms on Gauss space.

Another problem of considerable interest for the Riesz transforms on $\R^n$ is {\it Problem 7} in \cite{Ban} which asks for the best constant $C_p$ in the weak-type inequality 
\begin{equation}\label{WeakRiesz}
\|R_jf\|_{L^{p, \infty}(\R^n)}=\sup_{\lambda>0}\big(\lambda^p |\{x\in \R^n: |R_jf|>\lambda\}|\big)^{1/p}\leq C_p\|f\|_{L^p(\R^n)}, 
\end{equation}
$1\leq p<\infty$, where $|E|$ denotes the Lebesgue measure of the set $E$. The space $L^{p, \infty}(\R^n)$ consists of all measurable functions $g$ for which the left hand side of \eqref{WeakRiesz} (with $g$ in place of $R_jf$) is finite.  Under a suitable renorming of $L^{p, \infty}(\R^n)$ (see \eqref{renorm} below) replacing the left hand side of \eqref{WeakRiesz}, the case of $1<p<\infty$ is solved by Os\c ekowski in \cite{O2}.  This provides bounds on $C_p$.   The case $p=1$ remains open and it is not even known if $C_1$ has a bound independent of the dimension $n$.   The problem of obtaining a constant $C_1$ independent of dimension goes back to Stein \cite{Ste1, Ste2}.  For the best available bound thus far (which is of order $\log(n)$), we refer the reader to Janakiraman \cite{Jan1}. 

When $n=1$, the problem reduces to obtaining the best weak-type constant for the Hilbert transform (conjugate function) $H$.  In this case it is known that  
\[ \|H\|_{L^p(\R^n)\to L^{p, \infty}(\R^n)} ={\left(\frac{1}{\pi}\int_{-\infty}^\infty \frac{{\left|\frac{2}{\pi}
\log{|t|}\right|}^p}{t^2 + 1} \mbox{d}t\right)}^{-1/p}, \,\,\,\,1\leq p\leq 2.\]
The case $p=1$, where 
\begin{equation}\label{catalan}
D_1=\frac{1+\frac{1}{3^2}+\frac{1}{5^2}+\frac{1}{7^2}+\frac{1}{9^2}+
\cdots}{1-\frac{1}{3^2}+\frac{1}{5^2}-\frac{1}{7^2}+\frac{1}{9^2}-\cdots}=\frac{\pi^2}{8\beta(2)}\approx 1.328434313301,
\end{equation}
with $\beta(2)$  the so called ``Catalan's" constant, is due to B. Davis \cite{D}.  The case $1<p\leq 2$ was studied by Janakiraman \cite{Jan2}.  The case $2<p<\infty$ remains open  even for the Hilbert transform.

Another natural  replacement for the $L^p$-inequalities for the Hilbert transform, singular integrals and Fourier multipliers when $p=1$, are the Zygmund \cite{Zyg} and Stein \cite{St} $LLogL$ inequalities.   Given that the Riesz transforms (and many other multipliers arising from projections of martingale transforms such as all those studied in the literature cited in the first paragraph above) are bounded in $L^p$  with constants which are $O(p)$, as $p\to\infty$, and $O(1/(p-1))$ as $p\to 1$, the classical argument of Zygmund \cite[Chapter XII]{Zyg1}  (see also \cite[p.~44]{Gra}) gives  that these operators have local $LLogL$ inequalities.  However, those general arguments do not provide very precise information on these constants. 

The literature on both weak-type inequalities and $LlogL$ inequalities is very large and in addition to to the work of Davis \cite{D} and  Janakiraman \cite{Jan2} on sharp weak-type inequalities we mention here the work of  Bennett \cite{Ben}, Aar\~ao and Jorge \cite{AO}, Laeng \cite{L}, Os\c ekowski \cite{O1, O2} and  Pichorides \cite{P}. Most relevant to our results here are the logarithmic and weak-type estimates established in Os\c ekowski \cite{O1,O2}, which motivate our next  results in this paper. For the rest of  the paper, $\Phi$, $\Psi$ denote the Young functions on $[0,\infty)$, given by the formulas
\begin{equation}\label{defPP}
 \Phi(t)=e^t-1-t,\qquad \Psi(t)=(t+1)\log(t+1)-t.
\end{equation}
These functions are conjugate to each other, in the sense that $\Phi'=(\Psi')^{-1}$. Next, for any $K>2/\pi$, define
\begin{equation}\label{defL}
L(K)=\frac{K}{\pi}\int_\R \frac{\Phi\left(\left|\frac{2}{\pi K}\log|t|\right|\right)}{t^2+1}\mbox{d}t.
\end{equation}
Furthermore, if $1<p<\infty$ and  $q=p/(p-1)$ is the conjugate exponent of  $p$, put
\begin{equation}\label{defCp}
 C_p=\begin{cases}
\displaystyle \left[\frac{2^{q+2}\Gamma(q+1)}{\pi^{q+1}}\sum_{k=0}^\infty \frac{(-1)^k}{(2k+1)^{q+1}}\right]^{1/q}& \mbox{if }1<p<2,\\ 
\displaystyle \left[\frac{2^{q+2}\Gamma(q+1)}{\pi^{q}}\sum_{k=0}^\infty \frac{1}{(2k+1)^{q}}\right]^{1/q},& \mbox{if }2\leq p<\infty.
\end{cases}
\end{equation}
This constant can be written as 
$$
C_p=\frac{2}{\pi}\left[\frac{4}{\pi}\,\Gamma(q+1)\beta(q+1)\right]^{1/q},  \,\,\,\, 1\leq p<2, 
$$
where $\beta(q)$ is the Dirichlet beta function ($\beta(2)$ is the Catalan's constant in \eqref{catalan}) and 
$$
C_p= \left[\pi^{-q}(2^{q+1}-2)\Gamma(q+1)\zeta(q)\right]^{1/q}, \,\,\,\, 2\leq p<\infty,
$$
where $\zeta(q)$ is the Riemann zeta function.

For $f:\R^n\to \R$, let
\begin{equation}\label{renorm}
|||f|||_{L^{p,\infty}(\R^n)}=\sup\left\{|A|^{-1+1/p }\int_A |f|\mbox{d}x\;:\; A\in \mathcal{B}(\R^n),\,0<|A|<\infty\right\}
\end{equation}
denote the weak $p$-th norm of $f$, $1<p<\infty$. See Grafakos \cite[Chapter I]{Gra} for many properties of this norm and its connections  to the quantity on the left hand side of \eqref{WeakRiesz}.  In particular, note that with $A=\{x\in \R^n: |f(x)|>\lambda\}$, we immediately obtain that $\|f\|_{L^{p,\infty}(\R^n)}\leq |||f|||_{L^{p,\infty}(\R^n)}$.  

The principal $LLogL$ and weak-type results in \cite{O1} and \cite{O2} are the following: 
%\begin{theorem}\label{add}
Let $n$ be a fixed positive integer and let $j\in \{1,\,2,\,\ldots,\,n\}$.
\begin{itemize}
\item[(i)] For any $K>2/\pi$ and any $f:\R^n\to \R$ with $\int_{\R^n} \Psi(|f|)<\infty$ we have
\begin{equation}\label{add1}
 \int_A |R_jf(x)|\mbox{d}x\leq K\int_{\R^n}\Psi(|f(x)|)\mbox{d}x+L(K)\cdot|A|.
\end{equation}
\item[(ii)] We have
\begin{equation}\label{add2}
 |||R_jf|||_{L^{p,\infty}(\R^n)}\leq C_p||f||_{L^p(\R^n)},\qquad 1<p<\infty.
\end{equation}
%Both inequalities are sharp.
\end{itemize}

Both inequalities here are sharp.  The inequality  \eqref{add1} should be compared with the results in Pichorides \cite{P} for the Hilbert transform (conjugate function) on $\mathbb{S}^1$. 
In this paper we extend the logarithmic inequality \eqref{add1} and the weak-type inequality \eqref{add2}  to:
\bigskip

\begin{enumerate}
\item[(1)] Riesz transforms on Lie groups. The new results are Theorems  \ref{vR}, \ref{group} 
and \ref{group1}. 
\bigskip

\item[(2)] Riesz transforms on spheres  in $\R^n$. The new results are Theorems \ref{sphere1} and  \ref{sphere2}. 
\bigskip

\item[(3)] Riesz transforms on Gauss space. The new result is Theorem \ref{gauss1}. 
\end{enumerate} 
\bigskip 
Our proofs  rest on the probabilistic approach using differentially subordinate martingales which has been employed very effectively elsewhere (\cite{BO}, \cite{BW}, \cite{GV}, \cite{O1} and \cite{O2}, to cite but a few references) for similar problems. For Theorem \ref{manifolds2}, we follow the argument of \cite{BanBau} and apply Theorem \ref{thmmart} in place of Theorem 2.5 from that paper.  For Theorem \ref{manifolds1}, we simply use the probabilistic representation for the Riesz transforms given in \cite[Theorem 3.2]{Li}) with the corrected modification pointed out in \cite{BanBau} (as already discussed above) and again apply the new inequality \eqref{mart_inequality} of Theorem \ref{thmmart}.  (See also \cite{Li4} where Li elaborated further on the corrections.)  The same applies to Theorem \ref{Li-forms}.   Since these details amount to setting up the notation to apply Theorem \ref{thmmart}, we leave this to the reader.  For our results on Lie groups, spheres and Gauss space, 
we follow the presentation of Arcozzi \cite{Ar}.  
Once the inequalities are obtained on spheres, using Poincar\'e's observation that the Gaussian measure is obtained from the surface measure of the sphere by a limiting argument, we will deduce the corresponding bounds for Riesz transforms associated with the Ornstein-Uhlenbeck semigroup on Gauss space. This approach is very ``hands on" and conceptually interesting requiring several explicit computations.

The rest of the paper is organized as follows. In the next section, \S\ref{prob},  we present several sharp new inequalities for martingales which are the key to our applications. In \S \ref{manifoldtransforms} we derive logarithmic and weak-type inequalities for martingale transforms on manifolds. The next three sections, \S\ref{LieRiesz}, \S\ref{SphereRiesz} and \S\ref{GaussRiesz}, are devoted to the study of logarithmic and weak-type inequalities for Riesz transforms on compact Lie groups, spheres and Gauss space.

\section{New sharp $L^p$, logarithmic, and weak-type martingale inequalities}\label{prob}As announced above, our approach depends heavily on martingale methods. The purpose of this section is to introduce the appropriate machinery. For the sake of convenience, we have decided to split this section into four parts.

We begin with the necessary probabilistic background. Assume that $(\Omega,\mathcal{F},\mathbb{P})$ is a complete probability space, equipped with $(\F_t)_{t\geq 0}$, a nondecreasing family of sub-$\sigma$-fields of $\F$, such that $\F_0$ contains all the events of probability $0$. Let $X$, $Y$ be two adapted martingales taking values in $\R^n$. As usual, we assume that the processes have right-continuous trajectories with the limits from the left, i.e., {\cadlag}.  The symbol $[X,Y]$ denotes the quadratic covariance process of $X$ and $Y$; consult e.g. Dellacherie and Meyer \cite{DM} for details in the one-dimensional case, and extend the definition to the vector setting by $[X,Y]=\sum_{k=1}^n [X^k,Y^k]$, where $X^k$, $Y^k$ are the $k$-th coordinates of $X$, $Y$, respectively. Following Ba\~nuelos and Wang \cite{BW} and Wang \cite{W}, we say that $Y$ is \emph{differentially subordinate} to $X$, if $|Y_0| \leq |X_0|$ and the process $([X,X]_t-[Y,Y]_t)_{t\geq 0}$ is nonnegative and nondecreasing as a function of $t$. This definition of differential subordination is slightly more general that the original definition given by Burkholder, see for example \cite{Bu3}. In addition, we say that martingales $X$, $Y$ are \emph{orthogonal}, if $d[X^i,Y^j]=0$ (i.e., the process $[X^i,Y^j]$ is constant) for all $i,\,j$. 
We note that when the martingales have continuous paths, it is customary to write  $\langle X, Y\rangle$ for $[X,Y]$ and $\langle X\rangle$ for $[X,X]$. To be consistent  in our notation, we will simply write $[X, Y]$ and $[X, X]$ for both continuous and {\cadlag} martingales.  In addition, unless it is explicitly stated, our martingales are only assumed to be 
{\cadlag}. 

An  important example of differentially subordinated martingales with continuous paths arises as follows.  Let $B$ be a $n$-dimensional Brownian motion and $H$, $K$ two predictable processes with values in $\R^n$ such that $|K_t|\leq |H_t|$ for all $t\geq 0$.  If we define $X$, $Y$ by the stochastic integrals
$$ X_t=\int_{0+}^t H_s\cdot\mbox{d} B_s,\qquad Y_t=\int_{0+}^t K_s\cdot\mbox{d}B_s,\qquad t\geq 0$$ then 
$Y$ is differentially subordinate to $X$. If, in addition, we have $ H_t\cdot K_t=0$ for all $t$, then both processes are orthogonal. These facts follow immediately from the identities
$$ [X,Y]_t=\int_{0+}^t H_s\cdot K_s\, ds\qquad \mbox{and}\qquad  [X, X]_t- [Y, Y]_t=\int_{0+}^t |H_s|^2-|K_s|^2\mbox{d}s.$$

The differential subordination implies many interesting inequalities involving the martingales $X$, $Y$. The literature on the subject is very large, we refer the interested reader to the survey \cite{Bu3} by Burkholder or the monograph \cite{O0} by the second-named author. We will only focus on a few results which will be important for us in our further considerations in this paper. 

To study the estimates for the vector Riesz transforms, one needs good bounds for differentially subordinated martingales (without the orthogonality property). For instance, to establish the $L^p$-bound
$||R^{\R^n}f||_{L^p(\R^n)}\leq 2(p^*-1)||f||_{L^p(\R^n)}$
and 
$||R^{G}f||_{L^p(G)}\leq 2(p^*-1)||f||_{L^p(G)}$
for $1<p<\infty$ for Riesz transforms on $\R^n$ and on Lie groups, Ba\~nuelos and Wang \cite{BW} and Arcozzi \cite{Ar} exploited the celebrated inequalities  of Burkholder \cite{Bu1,Bu2} (see also Wang \cite{W}).  
\begin{theorem}\label{martingale_thm?}
If $X$, $Y$ are two $\R^n$-valued martingales such that $Y$ is differentially subordinate to $X$, then 
\begin{equation}\label{martingaleIn_thm?}
||Y||_p\leq (p^*-1)||X||_p,\qquad 1<p< \infty,
\end{equation}
and the constant is the best possible.
\end{theorem}

In order to apply martingale inequalities to manifolds and obtain results as in Theorem \ref{manifolds1}, we will prove here the following extension of Burkholder's theorem. 

\begin{theorem}\label{thmmart} Let $X$ and $Y$ be $\R^n$-valued martingales with continuous paths such that $Y$ is differentially subordinate to $X$.
Consider 
the solution of the matrix equation
$$\mbox{d}\mathcal{M}_t = \mathcal{V}_t\mathcal{M}_t\mbox{d}t,\qquad  \mathcal{M}_0 = \mbox{Id},$$
where $(\mathcal{V}_t)_{t\geq 0}$ is an adapted and continuous process taking values in the set of symmetric and 
non-positive $n\times n$ matrices. For a given $a\geq 0$, consider the process 
$$Z_t =e^{-at}\mathcal{M}_t\int_0^t e^{as}\mathcal{M}_s^{-1}\mbox{d}Y_s.$$
Then for any $1<p<\infty$ and $T\geq 0$ we have the sharp bound
\begin{equation}\label{mart_inequality}
||Z_T||_p\leq (p^*-1)||X_T||_p,
\end{equation}
\end{theorem}

This theorem is motivated by Theorems 2.5 and 2.6 in \cite{BanBau}, which concern slightly different type of estimates involving the square brackets of appropriate martingales. For the sake of completeness, we will establish below sharp versions of those theorems as well (though we will not need them in our study of Riesz transforms - however, the results are interesting on their own right). We need some notation. 

For $0<p<\infty$, let $A_p$, $D_p$ be the best constants in the following inequalities for the stopped Brownian motion: for any $\tau\in L^{p/2}$,
\begin{equation}\label{defAp}
 \left\|\sup_{0\leq s\leq \tau}|B_{\tau}|\right\|_p\leq A_p\|\tau^{1/2}\|_p
\end{equation}
and
\begin{equation}\label{Davis}
\|B_{\tau}\|_p\leq D_p\|\tau^{1/2}\|_p,\,\,\,\, 0<p<\infty.
\end{equation} 

We will prove the following statements. 

\begin{theorem}\label{potential-nonmax} 
Let $Y$ be an $\R^n$-valued martingale with continuous paths and let $a$, $\mathcal{M}_t$ and $Z_t$ be as in the statement of Theorem \ref{thmmart}. 
 Then for any $0<p<\infty$ and $T\geq 0$,  we have 
\begin{equation}\label{nonmax}
\|Z_T\|_p\leq D_p\|[Y,Y]_T^{1/2}\|_p.
\end{equation}
The estimate \eqref{nonmax} is sharp, as it is already sharp in the case $a=0$ and $\mathcal{V}\equiv 0$. 
\end{theorem}

The maximal version of the above result reads as follows. Unfortunately, we have managed to prove it only in the real-valued case.

\begin{theorem}\label{potential-Thm}
Let $Y$ be a real-valued martingale with continuous paths. Consider the process
$$Z_t=e^{-at+\int_0^t V_s ds}\int_0^t e^{as-\int_0^s V_udu}dY_s,$$
where $a\geq 0$ and $(V_t)_{t\geq 0}$ is a non-positive adapted and continuous process. Then for every $1 \leq p < \infty$ and any $T\geq 0$ we have the sharp bound
$$  \left\|\sup_{0\leq t\leq T}|Z_t|\right\|_p\leq A_p\|[Y, Y]^{1/2}_T\|_p.$$
\end{theorem}

\begin{remark}\label{remarkdavis} A few comments on the constants $A_p$ and $D_p$ are in order. As shown by Davis \cite{Dav1}, for $0<p\leq 2$ the constant $D_p$ is the smallest positive zero of the confluent hypergeometric function of parameter $p$, while for $p\geq 2$, it is equal to the largest positive zero of the parabolic cylinder function of parameter $p$ (for the necessary definitions, see \cite{AS} or below). While the constant $A_p$ is not known explicitly, its behavior as $p\to \infty$ can be easily determined.  Indeed, it follows from the sharp good-$\lambda$ inequality in \cite{Ban1} that $A_p=O(\sqrt{p})$, as $p\to\infty$.  Since by Doob's maximal inequality, 
$$
\left\|\sup_{0\leq t\leq \tau}|B_{t}|\right\|_p\leq \frac{p}{p-1}\|B_\tau\|, \,\,\,\, 1<p<\infty, 
$$
a better (and more explicit) uniform estimate can be obtained by combining this with $D_p$: $A_p\leq \frac{p}{p-1}D_p$ for $1<p<\infty$. But for $p\geq 2$ we have $D_p\leq 2\sqrt{p+1/2}$; see  \cite{Ban1}.  The better estimate  $D_p\leq 2\sqrt{p}$ valid for all $p\geq 1$ is proved in \cite{Car}.  Thus for $1< p<\infty$, $A_p\leq \frac{2p^{3/2}}{p-1}$.   
\end{remark}

We come back to martingale inequalities which will have direct implications for Riesz transforms. To study the logarithmic and weak-type bounds, we will require the following two statements. Recall the function $\Psi$ given in \eqref{defPP}.

\begin{theorem}\label{log_thm}
Let $X$, $Y$ be two $\R^n$-valued martingales such that $Y$ is differentially subordinate to $X$. Then for $K>1$ and any $E\in \F$,
\begin{equation}\label{LlogL}
\sup_{t\geq 0}\E |Y_t|1_E\leq K\sup_{t\geq 0}\E \Psi(|X_t|)+\frac{\mathbb{P}(E)}{2(K-1)}.
\end{equation}
For each $K$, the constant $1/(2(K-1))$ is the best possible.
\end{theorem}

\begin{theorem}\label{weak_thm}
Let $X$, $Y$ be two $\R^n$-valued martingales such that $Y$ is differentially subordinate to $X$. Then for $K>1$ and any $E\in \F$,
\begin{equation}\label{weak}
\sup_{t\geq 0}\E |Y_t|1_E\leq K_p||X||_p\mathbb{P}(E)^{1-1/p},
\end{equation}
where
\begin{equation}\label{defKp}
 K_p=\begin{cases}
\displaystyle \left(\frac{1}{2}\Gamma\left(\frac{2p-1}{p-1}\right)\right)^{1-1/p} & \mbox{if }1<p<2,\\
\displaystyle \left(p^{p-1}/2\right)^{1/p} & \mbox{if }p\geq 2.
\end{cases}
\end{equation}
For each $1<p<\infty$ the constant $K_p$ is the best possible.
\end{theorem}

On the other hand, if one is interested in bounds for directional Riesz transforms, one exploits differentially subordinate martingales satisfying the orthogonality property. For example, the following result of Ba\~nuelos and Wang \cite{BW} leads to sharp $L^p$-bounds for Riesz transforms on $\R^n$.  (See also Arcozzi \cite{Ar} for results on Lie groups).
\begin{theorem}\label{martingale_thm1}
Let $X$, $Y$ be two real-valued orthogonal martingales such that $Y$ is differentially subordinate to $X$. Then 
\begin{equation}\label{martingaleIn_thm1}
||Y||_p\leq \cot\left(\frac{\pi}{2p^*}\right)||X||_p,\qquad 1<p< \infty,
\end{equation}
and the constant is the best possible.
\end{theorem}

Thus, to establish logarithmic and weak-type inequalities for directional Riesz transforms, one needs ``orthogonal'' versions of Theorems \ref{log_thm} and \ref{weak_thm}. Unfortunately, we have been unable to establish such results. To overcome this difficulty, we will exploit the following dual statements, which have been obtained by the second-named author in \cite{O1} and \cite{O2}. Recall the function $\Phi$ given in \eqref{defPP} and the constant $L(K)$ given by \eqref{defL}.

\begin{theorem}\label{martinth}
Suppose that $X$, $Y$ are orthogonal martingales such that $||X||_\infty\leq 1$, $Y$ is differentially subordinate to $X$ and $Y_0\equiv 0$. Then for any $K>2/\pi$ we have
\begin{equation}\label{martingale_thm2}
\sup_{t\geq 0}\E \Phi\left(|Y_t|/K\right)\leq \frac{L(K)||X||_1}{K}.
\end{equation}
The inequality is sharp.
\end{theorem}

The second result, dual to the weak type estimate, is as follows (cf. \cite{O2}). 

\begin{theorem}\label{martingale_thm3}
Assume that $X$, $Y$ are orthogonal martingales such that $Y$ is differentially subordinate to $X$ and $Y_0\equiv 0$. Then for any $1< q<\infty$ we have
\begin{equation}\label{mainin4}
 ||Y||_q\leq C_p||X||_1^{1/q}||X||_\infty^{1/p}
\end{equation}
where $C_p$ is given by \eqref{defCp}.  The constant cannot be improved.
\end{theorem}

%%%%%%%%%%%%%%%%%%%%%%%%%%%%%%%%%%%%%%%%%%%%%%%%%%%%%

\subsection{Proof of Theorem \ref{thmmart}}  
The proof of this statement in the cases $1<p< 2$ and $p\geq 2$ will be completely different. In both cases, we will make use of Burkholder's special function corresponding to his celebrated $L^p$-inequalities \eqref{martingaleIn_thm?} for differentially subordinate martingales. However, in the first case we will exploit the integration argument (see \cite{O-2}, \cite{O-1}, \cite{O0}), while in the second case we will proceed directly; this approach will allow us to avoid several technical problems. Clearly, all we need is to establish the inequality \eqref{mart_inequality}; its sharpness follows immediately from the fact that the constant $p^*-1$ is the best in the bound $||Y||_p\leq (p^*-1)||X||_p$, which corresponds to the choice $\mathcal{V}\equiv 0$. Furthermore, observe that we may assume that $a=0$, replacing $\mathcal{V}$ by the symmetric and non-positive matrix $\mathcal{V}-a\mbox{Id}$, if necessary.

\begin{proof}[Proof of \eqref{mart_inequality}, $1<p<2$]
It is convenient to split the reasoning into two parts.

\smallskip

\emph{Step 1.} 
Fix $0<r<\infty$. We will exploit the special function $u_r:\R^n\times \R^n\to \R$, given by the formula
\begin{equation}\label{defur}
 u_r(x,y)=\begin{cases}
r^{-2}(|y|^2-|x|^2) & \mbox{if }|x|+|y|\leq r,\\
1-2r^{-1}|x| & \mbox{if }|x|+|y|>r.
\end{cases}
\end{equation}
It is straightforward to check the pointwise bound
\begin{equation}\label{intt}
u_r(x,y)\leq 1-2r^{-1}|x|\qquad \mbox{for }x,\,y\in \R^n.
\end{equation}
Introduce the stopping time $\tau=\inf\{t\geq 0:|Z_t|+|X_t|\geq r\}\wedge T$ and let 
$$\sigma_n=\inf\{t:|Y_t|+|X_t|\geq n\}\wedge T,\qquad n=1,\,2,\,\ldots$$
be a common localizing sequence for $X$ and $Y$ (here and below, we use the convention $\inf\emptyset=\infty$). First, we will prove that
\begin{equation}\label{I1}
 \E u_r(X_{\sigma_n},Z_{\sigma_n})\leq \E u_r(X_{\sigma_n\wedge \tau},Z_{\sigma_n\wedge \tau}).
\end{equation}
To show this, note that $u_r(X_{\sigma_n},Z_{\sigma_n})= u_r(X_{\sigma_n\wedge \tau},Z_{\sigma_n\wedge \tau})$ on the set $\{\tau=T\}$, and hence  
$ \E \left[u_r(X_{\sigma_n},Z_{\sigma_n})|\F_{\sigma_n\wedge \tau}\right]=u_r(X_{\sigma_n\wedge \tau},Z_{\sigma_n\wedge \tau})$ 
there. On the other hand, on $\{\tau<T\}$ we have, by \eqref{intt},
\begin{align*}
 \E \left[u_r(X_{\sigma_n},Z_{\sigma_n})|\F_{\sigma_n\wedge \tau}\right]&\leq 1-2r^{-1}\E(|X_{\sigma_n}||\F_{\sigma_n\wedge \tau})\\
 &\leq 1-2 |X_{\sigma_n\wedge \tau}|=u_r(X_{\sigma_n\wedge \tau},Z_{\sigma_n\wedge \tau}).
 \end{align*}
Adding the latter two facts and taking expectation yields \eqref{I1}. Now we apply It\^o's formula to the function $u_r$ and the process $(X_t,Z_t)_{0\leq t\leq \sigma_n\wedge \tau}$. Note that if $\tau>0$, then the process evolves in the set $\{(x,y):|x|+|y|\leq r\}$, in the interior of which $u_r$ is of class $C^\infty$.  Thus the use of It\^o's formula is permitted. We easily check that $Z$ satisfies the stochastic differential equation $$\mbox{d}Z_t=\mathcal{V}_tZ_t\mbox{d}t+\mbox{d}Y_t$$
(recall that we have assumed $a=0$) and we get
\begin{equation}\label{ito12}
 u_r(X_{\sigma_n\wedge \tau},Z_{\sigma_n\wedge \tau})=I_0+I_1+I_2+I_3,
\end{equation}
where
\begin{align*}
I_0&=u_r(X_0,Z_0),\\
I_1&=\int_{0+}^{\sigma_n\wedge \tau} 2\langle Z_s, \mathcal{V}_sZ_s\rangle\mbox{d}s,\\
I_2&=[Z,Z]_{\sigma_n\wedge \tau}-[Z,Z]_0-([X,X]_{\sigma_n\wedge \tau}-[X,X]_0),\\
I_3&=-2\int_0^{\sigma_n\wedge \tau} X_s\cdot \mbox{d}X_s+2\int_0^{\sigma_n\wedge \tau} Z_s\cdot \mbox{d}Y_s.
\end{align*}
The symbol $\langle \cdot,\cdot\rangle$ in $I_1$ denotes the usual scalar product in $\R^n$. Let us analyze the terms $I_0-I_3$. We start from observing that $I_0=u_r(X_0,0)\leq 0$. Next, since $\mathcal{V}$ takes values in the class of non-positive matrices, we see that the integrand in $I_2$ is nonpositive, and hence $I_1\leq 0$. To deal with $I_2$, note that 
$$ [Z,Z]_{\sigma_n\wedge \tau}-[Z,Z]_0=[Y,Y]_{\sigma_n\wedge \tau}-[Y,Y]_0\leq [X,X]_{\sigma_n\wedge \tau}-[X,X]_0,$$
where the latter bound follows from the differential subordination of $Y$ to $X$. Finally, both stochastic integrals in $I_3$ have mean zero. 
Therefore, integrating both sides of \eqref{ito12} gives $ \E u_r(X_{\sigma_n\wedge \tau},Z_{\sigma_n\wedge \tau})\leq 0$, which combined with \eqref{I1} yields
$$ \E u_r(X_{\sigma_n},Z_{\sigma_n})\leq 0.$$

\emph{Step 2.} We turn to the inequality \eqref{mart_inequality}. It is not difficult to check that the function
\begin{equation}\label{integf}
 U_p(x,y)=\frac{p^{3-p}(p-1)(2-p)}{2}\int_0^\infty r^{p-1}u_r(x,y)\mbox{d}r.
\end{equation}
admits the following explicit formula:
$$ U_p(x,y)=p^{2-p}(|y|-(p-1)^{-1}|x|)(|x|+|y|)^{p-1}.$$
This is the celebrated Burkholder's special function \cite{Bu2, Bu3}.  
By the previous step and Fubini's theorem, we have $ \E U_p(X_{\sigma_n},Z_{\sigma_n})\leq 0.$ However, $U_p$ satisfies the majorization
$$ |y|^p-(p-1)^{-p}|x|^p\leq U_p(x,y)$$
(as shown by Burkholder \cite{Bu2, Bu3}), so we get
$$ \E |Z_{\sigma_n}|^p\leq (p-1)^{-p}\E |X_{\sigma_n}|^p\leq (p-1)^{-p}||X_T||_p^p.$$
It remains to let $n\to \infty$ to obtain the claim.
 \end{proof}

We turn to the case $p\geq 2$. We would like to point out that the above approach does not work. Though there exist appropriate ``simple'' functions $u_r$, they lead to Burkholder's function
\begin{equation}\label{Up}
 \tilde{U}_p(x,y)=p(1-1/p)^{p-1}(|y|-(p-1)|x|)(|x|+|y|)^{p-1}
\end{equation}
which is \emph{not} sufficient for our purposes; see the remark after the property (d) below.

We will work with the following modification of $\tilde{U}_p$. Define $U=U_p:\R^n\times \R^n\to \R$ by the formula
$$ U_p(x,y)=\begin{cases}
p(1-1/p)^{p-1}(|y|-(p-1)|x|)(|x|+|y|)^{p-1} & \mbox{if }|y|\geq (p-1)|x|,\\
|y|^p-(p-1)^p|x|^p & \mbox{if }|y|<(p-1)|x|.
\end{cases}$$
  
We will need the following three properties of the function $U_p$, established by Burkholder \cite{Bu2} (see also Wang \cite{W}):

\begin{itemize}
\item[(a)] The function $U_p$ is of class $C^1$.

\item[(b)] We have the majorization
$$|y|^p-(p-1)^p|x|^p\leq U_p(x,y)\qquad \mbox{ for all }x,\,y\in \R^n.$$

\item[(c)] If $|x||y|\neq 0$ and $|y|\neq (p-1)|x|$, then for all $h,\,k\in \R^n$,
\begin{equation}\label{convvv}
\begin{split}
\langle hU_{pxx}(x,y),h\rangle+2\langle hU_{pxy}(x,y),k\rangle+\langle kU_{pyy}(x,y),k\rangle \leq c(x,y)(|k|^2-|h|^2),
\end{split}
\end{equation}
where $c$ is a nonnegative function given by
$$ c_p(x,y)=\begin{cases}
p(p-1)(|x|+|y|)^{p-2} & \mbox{if }|y|>(p-1)|x|,\\
p(p-1)^p|x|^{p-2} & \mbox{if }|y|<(p-1)|x|.
\end{cases}$$
\end{itemize}

Here, of course, $U_{pxx}$ denotes the second derivative of $U_p$ with respect to the variable $x$ (i.e., the $d\times d$ matrix which has the corresponding second-order partial derivatives as its entries); the matrices $U_{pxy}$ and $U_{pyy}$ are defined similarly. 

In our considerations below, the following property will also play a role. Since $U_p$ depends on $y$ only through the norm $|y|$, we get that $U_{py}(x,y)=\alpha(x,y) y$ for a certain $\alpha(x,y)\in \R$. The key fact is that $\alpha$ is nonnegative; summarizing, we have
\begin{itemize}
\item[(d)] $U_{py}(x,y)=\alpha(x,y) y$ for a certain $\alpha(x,y)\geq 0$. 
\end{itemize}

This condition is \emph{not} satisfied by the function $\tilde{U}_p$ given in \eqref{Up}: the corresponding $\alpha$ may take negative values. This is the reason why we have taken the slightly more complicated function $U_p$.

\begin{proof}[Proof of \eqref{mart_inequality}, $2\leq p<\infty$]

Consider a $C^\infty$ function $g:\R^n\times \R^n\to [0,\infty)$, supported on the unit ball of $\R^n\times \R^n$ and satisfying $\int_{\R^n\times \R^n}g=1$. Fix $\delta>0$ and define $U^\delta$ by the convolution
$$ U^\delta(x,y)=\int_{\R^n\times \R^n} U_p(x+\delta u,y+\delta v)g(u,v)\mbox{d}u\mbox{d}v.$$
Obviously, this new function is of class $C^\infty$. 
By integration by parts and (a), we see that the following formulas hold true:
\begin{equation}\label{Uy}
U^\delta_y(x,y)=\int_{\R^n\times \R^n} U_{py}(x+\delta u,y+\delta v)g(u,v)\mbox{d}u\mbox{d}v,
\end{equation}
\begin{align*}
U^\delta_{xx}(x,y)&=\int_{\R^n\times \R^n} U_{pxx}(x+\delta u,y+\delta v)g(u,v)\mbox{d}u\mbox{d}v,
\end{align*}
and similarly for $U^\delta_{xy}$ and $U^\delta_{yy}$. Consequently, we have that \eqref{convvv} holds true for $U^\delta$, with 
$$ c^\delta(x,y)=\int_{\R^n\times \R^n} c(x+\delta u,y+\delta v)g(u,v)\mbox{d}u\mbox{d}v\geq 0.$$ 
Introduce the stopping times 
$$\sigma_n=\inf\{s: ||\mathcal{V}_s||+|X_s|+|Z_s|\geq n\}\wedge T,\qquad n=1,\,2,\,\ldots.$$
As we have already noted, $Z$ satisfies the equation $ \mbox{d}Z_t=\mathcal{V}_tZ_t\mbox{d}t+\mbox{d}X_t.$ Therefore, an application of It\^o's formula yields
\begin{equation}\label{ito13}
 U^\delta(X_{\sigma_n},Z_{\sigma_n})=U^\delta(X_0,Z_0)+I_1+I_2/2+I_3,
\end{equation}
where
\begin{align*}
I_1&=\int_0^{\sigma_n} \langle U_y^\delta(X_s,Z_s), \mathcal{V}_sZ_s\rangle\mbox{d}s,\\
I_2&=\int_0^{\sigma_n} U_{xx}^\delta(X_s,Z_s)\cdot \mbox{d}[X,X]_s\\
&\quad +2\int_0^{\sigma_n\wedge T} U_{xy}^\delta(X_s,Z_s)\cdot \mbox{d}[X,Z]_s+\int_0^{\sigma_n} U_{yy}^\delta(X_s,Z_s)\cdot \mbox{d}[Z,Z]_s,\\
I_3&=\int_0^{\sigma_n} U_x^\delta(X_s,Z_s)\cdot \mbox{d}X_s+\int_0^{\sigma_n} U_y^\delta(X_s,Z_s)\cdot \mbox{d}Y_s.
\end{align*}
Here in the definition of $I_2$ we have used a shortened notation; for instance, the first integral equals
$$ \sum_{i,j=1}^d\int_0^{\sigma_n} U_{x_ix_j}^\delta(X_s,Z_s) \mbox{d}[X^i,X^j]_s.$$

Let us analyze the terms $I_1$ through $I_3$ separately. To handle $I_1$, note that by (d), \eqref{Uy} 
and the fact that $||\mathcal{V}_s||\leq n$ for $s\in(0,\sigma_n]$, we get
\begin{align*}
\langle U_y^\delta(X_s,Z_s), \mathcal{V}_sZ_s\rangle&=\int_{\R^n\times \R^n} \big\langle U_y(X_s+\delta u,Z_s+\delta v),\mathcal{V}_s(Z_s+\delta v)\big\rangle g(u,v)\mbox{d}u\mbox{d}v\\
&\quad -\delta \int_{\R^n\times \R^n} \big\langle U_y(X_s+\delta u,Z_s+\delta v),\mathcal{V}_sv\big\rangle g(u,v)\mbox{d}u\mbox{d}v\\
&\leq n\delta \int_{\R^n\times \R^n} \big| U_y(X_s+\delta u,Z_s+\delta v)\big| g(u,v)\mbox{d}u\mbox{d}v\\
&\leq C(n,p)\delta.
\end{align*}
Here $C(n,p)$ is a certain constant depending only on the parameters indicated. 
Thus, we have $I_1\leq TC(n,p)\delta$. Next, using a simple approximation argument of Wang \cite{W} and \eqref{convvv}, we get
\begin{align*}
 I_2&\leq \int_0^{\sigma_n}c^\delta(X_s,Z_s)\,\mbox{d}([Z,Z]_s-[Y,Y]_s)\\
&=\int_0^{\sigma_n}c^\delta(X_s,Z_s)\,\mbox{d}([X,X]_s-[Y,Y]_s)\leq 0,
\end{align*}
where in the latter estimate we have exploited the differential subordination of $Y$ to $X$. Finally, both stochastic integrals in $I_3$ are equal to $0$. Plug all these facts into \eqref{ito13}, take expectation of both sides and let $\delta\to 0$. Since $U_p$ is continuous, we have that $U^\delta \to U_p$ pointwise; furthermore, the processes $Z$ and $X$ are bounded on the interval $(0,\sigma_n]$ which makes Lebesgue's dominated convergence theorem applicable. Consequently, we obtain $ \E U(X_{\sigma_n},Z_{\sigma_n})\leq \E U(X_0,Z_0)\leq 0,$ which by the majorization (b) implies
$$ \E |Z_{\sigma_n}|^p\leq (p-1)^p\E |X_{\sigma_n}|^p\leq (p-1)^p||X_T||_p^p.$$
Letting $n\to\infty$ yields the claim. This completes the proof of the theorem for all $1<p<\infty$.  
\end{proof}

\subsection{Proof of Theorem \ref{potential-nonmax}} 
We start from a few definitions; for the  detailed study of the objects below, we refer the interested reader to \cite{AS}. First we introduce Kummer's
function $M(a, b, z)$: it is a solution of the differential equation
$$ zw''(z) + (b- z)w'(z) - aw(z) = 0.$$
The explicit form of $M(a, b, z)$ is
\begin{equation}\label{kummer}
M(a,b,z)=1+\frac{a \cdot z}{b}+\frac{a(a+1)\cdot z^2}{b(b+1)\cdot 2!}+\frac{a(a+1)(a+2)\cdot z^3}{b(b+1)(b+2)\cdot 3!}+\ldots.
\end{equation}
Then $M_p$, the so-called confluent hypergeometric function, is given by the formula $M_p(x)=M(-\frac{p}{2},\frac{1}{2},\frac{x^2}{2})$. Let $\nu_p$ denote its smallest positive zero (the definition makes sense, see e.g. \cite{AS}). These objects allow us to define the special function corresponding to \eqref{nonmax}, in the range $0<p\leq 2$. Namely, for $x\in \R^n$ and $t\geq 0$, put
$$ U_p(x,t)=\begin{cases}
|x|^p-\nu_p^pt^{p/2} & \mbox{if }|x|\geq \nu_pt^{1/2},\\
p\nu_p^{p-1}t^{p/2}M_p(|x|/\sqrt{t})/M_p'(\nu_p) & \mbox{if }|x|<\nu_pt^{1/2}
\end{cases}$$
(there is no zero in the denominator, \eqref{kummer} gives that $M_p'$ takes negative values on $(0,\infty)$).

In the case $2\leq p<\infty$, we will require another special objects: \emph{parabolic cylinder functions}. They are related to the confluent hypergeometric functions as follows. First, put
$$ Y_1(x)=(2^{p/2}/\sqrt{\pi})\Gamma((p +1)/2)e^{-x^2/4}M\left(-\frac{p}{2},\frac{1}{2},\frac{x^2}{2}\right),$$
$$ Y_2(x)=(2^{(p+1)/2}/\sqrt{\pi})\Gamma((p +2)/2)xe^{-x^2/4}M\left(-\frac{p}{2}+\frac{1}{2},\frac{3}{2},\frac{x^2}{2}\right)$$
and define the parabolic cylinder function $\mathcal{D}_p$ by
$$ \mathcal{D}_p(x)=Y_1(x)\cos\left(\frac{p\pi}{2}\right)+Y_2(x)\sin\left(\frac{p\pi}{2}\right).$$
We set $h_p(x)=e^{x^2/4}\mathcal{D}_p(x)$, $x\in \R$, and denote the largest positive zero of $h_p$ by  $\mu_p$ (this is well defined, see \cite{AS}). We are ready to introduce the special functions $U_p$ corresponding to \eqref{nonmax} in the range $2\leq p<\infty$. Define, for $x\in\R^n$ and $t\geq 0$,
$$ U_p(x,t)=\begin{cases}
|x|^p-\mu_p^pt^{p/2} & \mbox{if }|x|<\mu_pt^{1/2},\\
p\mu_p^{p-1}t^{p/2}h_p(|x|/\sqrt{t})/h_p'(\mu_p) & \mbox{if }|x|\geq \mu_pt^{1/2}
\end{cases}$$
(the definition makes sense: it was proved in Lemma 5.3 in \cite{W0} that the function $h_p'$ is strictly positive on $[\mu_p,\infty)$). 

We will prove the following. 
\begin{lemma}
For any fixed $0<p<\infty$, the function $U_p$ enjoys the following properties:
\begin{itemize}
\item[(a)] $U_p$ is of class $C^1$.
\item[(b)] We have the majorization
$$ |x|^p-D_p^pt^{p/2}\leq U_p(x,t)\qquad \mbox{for all }x\in \R^n,\,t\geq 0.$$
\item[(c)] If $t>0$ and $|x|\neq D_pt^{1/2}$, then for any $h\in \R^n$,
$$ \frac{1}{2}\langle hU_{pxx}(x,t),h\rangle+U_{pt}(x,t)|h|^2\leq 0.$$
\item[(d)] For any $x\in \R^n$ and $t>0$ we have $U_{px}(x,t)=\alpha(x,t)x$ for some $\alpha(x,t)\geq 0$.
\end{itemize}
\end{lemma}
\begin{proof} We establish the properties separately.

\smallskip

\emph{Proof of (a).}  This is straightforward; we leave the necessary calculations to the reader. 

\smallskip

\emph{Proof of (b).} The majorization was already proved by Davis \cite{Dav1} and Wang \cite{W0}. 

\smallskip

\emph{Proof of (c).} Assume first that $0<p<2$. If $|x|> \nu_pt^{1/2}$,  the inequality takes the form
$$ p(p-2)|x|^{p-4}\langle x,h\rangle^2+p|x|^{p-2}|h|^2-p\nu_p^pt^{p/2-1}|h|^2\leq 0.$$
However, the first term is nonpositive and it suffices to note that
$$ p\nu_p^pt^{p/2-1}|h|^2\geq p\nu_p^2|x|^{p-2}|h|^2\geq p|x|^{p-2}|h|^2.$$
If $|x|<\nu_pt^{1/2}$, then, after some tedious calculations, we rewrite the desired bound in the equivalent form
$$ (|x|^2|h|^2-\langle x,h\rangle^2)\left(M_p''\left(\frac{|x|}{\sqrt{t}}\right)\frac{|x|}{\sqrt{t}}-M_p'\left(\frac{|x|}{\sqrt{t}}\right)\right)\leq 0.$$
Therefore, it suffices to show that $uM_p''(u)-M_p'(u)\leq 0$ for $u\geq 0$. But this is easy: we have equality for $u=0$, and
$$ (uM_p''(u)-M_p'(u))'=uM_p'''(u)=-puM_{p-2}'(u)=-pu^2M'\left(\frac{2-p}{2},\frac{1}{2},\frac{u^2}{2}\right)\leq 0,$$
since all the terms in the series defining $M'(\frac{2-p}{2},\frac{1}{2},\frac{u^2}{2})$ are nonnegative.

We turn to the case $p\geq 2$. If $|x|<\mu_pt^{1/2}$, the estimate in (c) reads
$$ p(p-2)|x|^{p-4}\langle x,h\rangle^2+p|x|^{p-2}|h|^2-p\mu_p^pt^{p/2-1}|h|^2\leq 0.$$
Since $\langle x,h\rangle\leq |x||h|$ and $\mu_p^{p-2}t^{p/2-1}>|x|^{p-2}$, we will be done if we show that
$$ p(p-2)|x|^{p-2}|h|^2+p|x|^{p-2}|h|^2-p\mu_p^2|x|^{p-2}|h|^2\leq 0,$$
or $\mu_p^2\geq p-1$. However, the latter estimate appears in Lemma 5.4 in \cite{W0}. If $|x|>\mu_pt^{1/2}$, then, after some straightforward computations, we obtain the following  bound to prove:
$$ (|x|^2|h|^2-\langle x,h\rangle^2)\left(h_p''\left(\frac{|x|}{\sqrt{t}}\right)\frac{|x|}{\sqrt{t}}-h_p'\left(\frac{|x|}{\sqrt{t}}\right)\right)\geq 0.$$
The expression in the first parentheses is nonnegative, so it suffices to show that the second factor also has this property. We use the following statements which can be found in \cite{W0}: first, the function $h_p$ satisfies the differential equation $h_p''(u)-uh_p'(u)+ph_p(u)=0$; second, we have $h_p^{(3)}>0$ and $h_p'>0$ on $[\mu_p,\infty)$. The combination of these two facts gives
$$ 0<h_p^{(3)}(u)=uh_p''(u)-(p-1)h_p'(u)\leq uh_p''(u)-h_p'(u)$$
for $u\geq \mu_p$. The proof of (c) is finished.

\smallskip

\emph{Proof of (d).} It suffices to prove that for any fixed $t$, $U_p$ is an increasing function of $|x|$. But this follows immediately from the facts that for $0<p<2$ the function $M_p'$ is negative on $(0,\infty)$ (see the definition of $M_p$ and differentiate term-by-term), and for $p\geq 2$, the function $h_p$ is increasing on $[\mu_p,\infty)$ (cf. Lemma 5.3 in \cite{W0}).
\end{proof}

We are ready to establish \eqref{nonmax}, and the proof is similar to that of Theorem \ref{thmmart}: it exploits $U_p$ and a mollification argument. We may assume that $a=0$, replacing $\mathcal{V}$ by $\mathcal{V}-a$Id if this is not the case. Let $g:\R^n\times \R\to [0,\infty)$ be a $C^\infty$ function, supported on the unit ball of $\R^n\times \R$ and such that $\int_{\R^n\times \R}g=1$. For a fixed $\delta>0$, let $U^\delta:\R^n\times [\delta,\infty)\to \R$ be given by the convolution
$$ U_p^\delta(x,t)=\int_{[-1,1]^n\times [-1,1]} U_p(x+\delta u,t+\delta v)g(u,v)\mbox{d}u\mbox{d}v.$$
This function is of class $C^\infty$; furthermore, as we have already noted above, $Z$ satisfies the stochastic differential equation $dZ_t=\mathcal{V}_tZ_t\mbox{d}t+dY_t$. Introduce the stopping time
$$ \sigma_n=\inf\{t:||\mathcal{V}_t||+|Z_t|+[Z,Z]_t\geq n\}\wedge T$$
and apply It\^o formula to get
\begin{equation}\label{IIto}
 U_p^\delta(Z_{\sigma_n},\delta+[Z,Z]_{\sigma_n})=U_p^\delta(0,\delta)+I_1+I_2+I_3,
\end{equation}
where 
\begin{align*}
 I_1&=\int_0^{\sigma_n} U_{px}^\delta(Z_s,\delta+[Z,Z]_s)\cdot \mbox{d}Y_s,\\
 I_2&=\int_0^{\sigma_n} \frac{1}{2}U_{pxx}^\delta(Z_s,\delta+[Z,Z]_s)\cdot \mbox{d}[Y,Y]_s+\int_0^{\sigma_n}U_{pt}^\delta(Z_s,\delta+[Z,Z]_s)\mbox{d}[Y,Y]_s\\
 I_3&=\int_0^{\sigma_n} \langle U_{px}^\delta(Z_s,\delta+[Z,Z]_s),\mathcal{V}_sZ_s\rangle\mbox{d}s.
\end{align*}
The term $I_1$ has mean zero. The term $I_2$ is nonpositive, which can be shown with the use of (c) and the approximation argument of Wang \cite{W}. Finally, the term $I_3$ is dealt with in the same manner as the term $I_1$ in the proof of \eqref{mart_inequality}, $p\geq 2$: we have
\begin{align*}
\langle U_{px}^\delta&(Z_s,\delta+[Z,Z]_s),\mathcal{V}_sZ_s\rangle \leq C(n,p)\delta,
\end{align*}
for some $C(n,p)$ depending only on $n$ and $p$, so $I_3\leq C(n,p)T\delta$. Plugging all the facts above into \eqref{IIto} and taking expectation yields 
$$ \E U_p^\delta(Z_{\sigma_n},\delta+[Z,Z]_{\sigma_n})\leq U_p^\delta(0,\delta)+C(n,p)T\delta.$$
Letting $\delta\to 0$ gives $\E U_p(Z_{\sigma_n},[Z,Z]_{\sigma_n})\leq U_p(0,0)=0$, which, by (b), implies
$$ \E |Z_{\sigma_n}|^p\leq D_p^p\E [Z,Z]_{\sigma_n}^{p/2}.$$
It remains to let $n$ go to infinity, and the claim follows.

\subsection{Proof of Theorem \ref{potential-Thm}} The reasoning is similar as above, but the crucial difference is that the special function is not given explicitly.  Recall that for $0<p<\infty$, $A_p$ is the best constant in the Burkholder-Davis-Gundy inequality \eqref{defAp} for the stopped Brownian motion. 
Let $U$ be the value function of the corresponding optimal stopping problem: that is, for $x\in \R$, $y\geq 0$, $t\geq 0$, put
$$ U(x,y,t)=\sup_{\tau\in L^{p/2}} \E G\left(x+B_\tau,\left(\sup_{0\leq s\leq \tau}|x+B_s|\right)\vee y,t+\tau\right),$$
where the gain function $G$ is given by $G(x, y,t)=y^p-A_p^pt^{p/2}.$ Observe that $U$ satisfies the symmetry condition
\begin{equation}\label{symsU}
U(x,y,t)=U(-x,y,t),
\end{equation}
which follows immediately from the fact that $-B$ is also a Brownian motion. 
By the strong Markov property, one easily checks that the function $U$ satisfies the inequalities
\begin{equation}\label{Markov}
U_t+\frac{1}{2}U_{xx}\leq 0\qquad\mbox{and}\qquad  U_y(x,|x|,z)\leq 0.
\end{equation}
Finally, we have $U\geq G$, since one can always consider $\tau\equiv 0$ in the definition of $U$.

Next, let us establish the following property of $U$.

\begin{lemma}\label{aux}
If $p\geq 1$, then for any fixed $y$, $t$, the function $x\mapsto U(x,y,t)$ is convex.
\end{lemma}
\begin{proof}
Pick $x_1$, $x_2\in \R$, $\lambda\in (0,1)$ and $\tau\in L^{p/2}$. Put $x=\lambda x_1+(1-\lambda)x_2$. For any $s\geq 0$, we have
$$ \Big(|x+B_s|\vee y\Big)^p\leq \lambda \Big(|x_1+B_s|\vee y\Big)^p+(1-\lambda)\Big(|x_2+B_s|\vee y\Big)^p$$
and this inequality is preserved if we take the supremum over $0\leq s\leq \tau$ in all the three terms above. This yields
\begin{align*}
 &\E \Big[\Big(\sup_{0\leq s\leq \tau}|x+B_s|\vee y\Big)^p-A_p^p(t+\tau)^{p/2}\Big] \leq \lambda U(x_1,y,t)+(1-\lambda)U(x_2,y,t)
\end{align*}
and taking the supremum over all $\tau$ gives the claim.
\end{proof}

Having established Lemma \ref{aux}, we can now proceed with the proof of the theorem.  We assume, as we may, that $\E[Y,Y]_T^{p/2}<\infty$. The 
process $Z$ satisfies the stochastic differential equation  $dZ_t=Z_tV_t\mbox{d}t+dY_t.$ 
Apply It\^o's formula to $U$ and the process $R=(Z,\sup |Z|,[Y,Y])$ (we may assume that $U$ has the necessary regularity, using an appropriate mollification argument if necessary; see above). We obtain
$$ U\left(Z_{t},\sup_{0\leq s\leq t}|Z_s|,[Y,Y]_{t}\right)=I_0+I_1+I_2+I_3+I_4,$$
where
\begin{align*}
 I_0&=U(0,0, 0),\\
 I_1&=\int_0^{t} U_x(R_s)\mbox{d}Y_s,\\
 I_2&=\int_0^{t} \left[\frac{1}{2}U_{xx}(R_s)+U_t(R_s)\right]\mbox{d}[Y,Y]_s\\
 I_3&=\int_0^{t} U_y(R_s)\mbox{d}(\sup Z_s)\\
 I_4&=\int_0^t U_x(R_s)Z_sV_s\mbox{d}s.
 \end{align*}
 However, we have $I_0=U(0,0,0)\leq 0$, by the definition of $U$ and the fact that $A_p$ is the best constant in \eqref{defAp}. The term $I_1$ defines a local martingale and therefore, applying localization if necessary, we may assume that $\E I_1=0$. The terms $I_2$ and $I_3$ are nonpositive by \eqref{Markov}: for $I_2$  this is clear, for $I_3$ one needs to observe that the process $\sup Z$ increases on the (random) set $\{t: Z_t=\sup_{0\leq s\leq t} |Z_s|\}$, on which $U_y$ is nonpositive. It remains to deal with $I_4$. By Lemma \ref{aux} and the symmetry condition \eqref{symsU}, we see that for fixed $y,\,t$, $x\mapsto U(x,y,t)$ decreases on $(-\infty,0]$ and increases on $[0,\infty)$. Therefore, $U_x(R_s)$ has the same sign as $Z_s$, and this implies that the integrand in $I_4$ is nonpositive (since $V\leq 0$); so, $I_4\leq 0$. Thus,
 $$ \E U\left(Z_{\sigma_n}, \sup_{0\leq s\leq \sigma_n} |Z_s|,[Y,Y]_{\sigma_n}\right)\leq 0$$
 for some increasing sequence $(\sigma_n)_{n\geq 0}$ of stopping times converging to $T$ almost surely. 
 Since $U$ majorizes $G$, the same is true if we replace $U$ with $G$. Equivalently,
$$ \E \sup_{0\leq s\leq \sigma_n}|Z_s|^p\leq A_p^p\E [Y,Y]_{\sigma_n}^{p/2}.$$
It remains to let $n\to \infty$ to get the claim, by Lebesgue's monotone convergence theorem.

\subsection{Proof of Theorem \ref{log_thm}}
Once again, we shall deduce the inequality \eqref{LlogL} from the existence of a certain special function (or rather, a family of certain special functions) $U:\R^n\times \R^n\to \R$. To simplify the technicalities which arise during the study of the analytic properties of these special functions, we shall combine Burkholder's technique with the integration argument, which has already appeared in our considerations above. We  first introduce two simple functions  $u_1,\,u_\infty:\R^n\times \R^n\to \R$, for which the calculations are  easy, and then define $U$ by integrating these two functions against appropriate nonnegative kernels. Let
$$ u_1(x,y)=\begin{cases}
|y|^2-|x|^2 \,\,\,\qquad& \mbox{if }|x|+|y|\leq 1,\\
1-2|x| & \mbox{if }|x|+|y|> 1
\end{cases}$$ 
and
$$ u_\infty(x,y)=\begin{cases}
0 & \mbox{if }|x|+|y|\leq 1,\\
(|y|-1)^2-|x|^2 & \mbox{if }|x|+|y|>1.
\end{cases}$$
We have already encountered the function $u_1$ in \eqref{defur} (in fact, we have $u_r(x,y)=u_1(x/r,y/r)$ for all $r>0$ and $x,\,y\in \R^n$). 
These  functions enjoy the following property (see Lemma 2.2 in \cite{O-1}).

\begin{lemma}\label{mainmar}
For all  $\R^n$-valued martingales $X$, $Y$ such that $Y$ is differentially subordinate to $X$, we have
$$ \E v_1(X_t,Y_t)\leq 0\qquad \mbox{for all }t\geq 0.$$
If in addition $X$ satisfies $||X||_2<\infty$, then
$$ \E v_\infty(X_t,Y_t)\leq 0\qquad \mbox{for all }t\geq 0.$$
\end{lemma}

We are ready to define the special function corresponding to the logarithmic inequality \eqref{LlogL}. Let $U:\R^n\times \R^n\to \R$ be given by
\begin{equation}\label{defU1}
 U(x,y)=\int_0^\infty a(\lambda)u_1(x/\lambda,y/\lambda)\mbox{d}\lambda+\frac{1}{2(K-1)},
\end{equation}
where 
$$ a(\lambda)=\frac{K}{2}\left(\frac{\lambda}{\lambda+1}\right)^2\chi_{[(K-1)^{-1},\infty)}(\lambda).$$
A computation shows that $U$ admits the following explicit formula: we have
$$ U(x,y)=\frac{K-1}{2}(|y|^2-|x|^2)+\frac{1}{2(K-1)}$$
if $|x|+|y|\leq (K-1)^{-1}$, and
$$ U(x,y)=K|y|+(K-1)(|x|+1)-K-K(|x|+1)\log\left[\frac{K-1}{K}(|x|+|y|+1)\right]$$
if $|x|+|y|>(K-1)^{-1}.$ 
We will establish the following majorization.

\begin{lemma}\label{maj_lemma1}
For any $(x,y)\in \R^n\times \R^n$ we have 
\begin{equation}\label{maj1}
U(x,y)\geq \max\left\{|y|,\frac{1}{2(K-1)}\right\}-K\Psi(|x|).
\end{equation}
\end{lemma}
\begin{proof} Of course, it suffices to show the claim for $n=1$ and nonnegative $x$, $y$. 
Suppose first that $y\leq (2(K-1))^{-1}$. Note that for a fixed $x$, the function $y\mapsto u_1(x,y)$ is a nondecreasing on $[0,\infty)$ and hence, by \eqref{defU1}, $U$ also has this property. Therefore, we will be done if we show the majorization for $y=0$. If $x\leq 1/(K-1)$, the inequality takes the form
$ F(x)=-(K-1)x^2/2+K\Psi(x)\geq 0$.  This follows from 
$$ F(0)=F'(0+)=0 \quad \mbox{and}\quad F''(x)=-(K-1)+\frac{K}{x+1}\geq 0.$$
On the other hand, if $x>1/(K-1)$, the majorization is equivalent to
$$ (x+1)\left(K\log\frac{K}{K-1}-1\right)-K-\frac{1}{2(K-1)}\geq 0.$$
But the left-hand side is a nondecreasing function of $x$, and we have already proved the bound for $x=1/(K-1)$. This yields \eqref{maj1} for $y\leq (2(K-1))^{-1}$.

Now suppose that $y>(2(K-1))^{-1}$. It is easy to see that for a given $x\geq 0$, the function 
$ \xi_x(y)=U(x,y)-y+K\Psi(x)$ is convex on $[0,\infty)$ and satisfies 
$$\xi_x\left(\frac{x+1}{K-1}\right)=\xi_x'\left(\frac{x+1}{K-1}\right)=0.$$
This immediately yields the majorization. 
\end{proof}

Before we proceed, let us record here that both sides of \eqref{maj1} are equal on the set
\begin{equation}\label{defD}
 \mathcal{D}=\big\{(x,y):|y|=(|x|+1)/(K-1)\big\}.
\end{equation}
Later on, this fact will turn out to be useful.

\begin{proof}[Proof of \eqref{LlogL}] We may assume that $\E \Psi(|X_t|)<\infty$, since otherwise the claim is trivial. By Fubini's theorem and Lemma \ref{mainmar}, we see that
$$ \E U(X_t,Y_t)\leq \frac{1}{2(K-1)}.$$
Thus, an application of \eqref{maj1} yields 
$$ \E \max\left\{|Y_t|,\frac{1}{2(K-1)}\right\}\leq K\E \Psi(|X_t|)+\frac{1}{2(K-1)},$$ or, equivalently,
\begin{equation}\label{intin}
 \E \max\left\{|Y_t|-\frac{1}{2(K-1)},0\right\}\leq K\E \Psi(|X_t|).
\end{equation}
Now, for a given event $E\in \F$, let 
$$E^-=E\cap\{|Y_t|\leq (2(K-1))^{-1}\}\quad \mbox{ and  }\quad E^+=E\cap\{|Y_t|>(2(K-1))^{-1}\}.$$
We have $$\E |Y_t|1_{E^-}\leq \mathbb{P}(E^-)/(2(K-1))$$ and
$$ \E \left\{|Y_t|-\frac{1}{2(K-1)}\right\}1_{E^+}\leq \E \max\left\{|Y_t|-\frac{1}{2(K-1)},0\right\}\leq K\Psi(|X_t|).$$
Adding the last two inequalities yields \eqref{LlogL}.
\end{proof}

\begin{proof}[Sharpness]
We will show that the constant $1/(2(K-1))$ cannot be replaced by a smaller one, by picking $E=\Omega$ and considering the following one-dimensional example. Let $B=(B_t)_{t\geq 0}$ be a standard Brownian motion starting at $1/(2(K-1))$ and stopped upon exiting $[0,\infty)$. Consider the martingale  $D$ given by the stochastic integral
$$ D_t=\frac{1}{2(K-1)}+\int_0^t \operatorname*{sgn}B_s\;\mbox{d}B_s.$$
Then $D$ is differentially subordinate to $B$, since $[D,D]=[B,B]$. Directly from the definition, we see that if $(D,B)$ belongs to the first quadrant (i.e., $D>0$), then locally it moves along the line segment of slope $-1$; similarly, if $D<0$, then it evolves along the line of slope $+1$. Consequently, $(B,D)$ takes values in the set $$\mathcal{C}=\{(x,y):x\geq 0,\,x+|y|\geq 1/(2(K-1))\},$$ in the interior of which $U$ is of class $C^2$. Since 
$$U_{xx}(x,y)+2U_{xy}(x,y)\cdot \operatorname*{sgn}y+U_{yy}(x,y)=0\quad \mbox{and}\quad U_y(x,0)=0\quad \mbox{ on }\,\,\, \mathcal{C},$$
The It\^o-Tanaka formula implies  that the process $(U(B_t,D_t))_{t\geq 0}$ is a martingale. 

Recall now the set $\mathcal{D}$ given by \eqref{defD} and consider the stopping time
$$ \tau=\inf\{t\geq 0: (B_t,D_t)\in \mathcal{D}\}.$$
This stopping time is finite almost surely; in fact, it can be easily shown that $\E \tau^{p/2}<\infty$ for some $p>1$. Put $X_t=B_{t \wedge \tau}$ and $Y_t=D_{t \wedge \tau}$ for $t\geq 0$. Then $Y$ is differentially subordinate to $X$ and we have
$$ \E U(X_t,Y_t)=\E U(X_0,Y_0)=\frac{1}{2(K-1)}.$$
Since $\tau\in L^{p/2}$ and $U(x,y)\leq C(|x|^p+|y|^p+1)$ for some absolute constant $C$, we may let $t\to\infty$ to obtain
  $$ \E U(X_\infty,Y_\infty)=\frac{1}{2(K-1)}.$$
However, the terminal value $(X_\infty,Y_\infty)$ belongs to $\mathcal{D}$, and hence
 $$ U(X_\infty,Y_\infty)=|Y_\infty|-K\Psi(|X_\infty|),$$
 almost surely; see the end of the proof of Lemma \ref{maj_lemma1}. It suffices to plug this into the previous identity and use the equalities
$$ \sup_{t\geq 0}\E |Y_t|=\E |Y_\infty|,\qquad \sup_{t\geq 0}\E \Psi(|X_t|)=\E \Psi(|X_\infty|),$$
to get the desired sharpness.
\end{proof}

%%%%%%%%%%%%%%%%%%%%%%%%%%%%%%%%%%%%%%%%%%%%%%%%%%%%%%%%%%%%%%%%%%%%%

\subsection{Proof of Theorem \ref{weak_thm}--case $1<p\leq 2$} Here the reasoning is much more complicated. Let us first handle the simple case $p=2$. An application of Schwarz inequality gives
$$ \E |Y_t|1_E\leq ||Y_t||_2\mathbb{P}(E)\leq ||X||_2\mathbb{P}(E)^{1/2},$$
so \eqref{weak} follows. The sharpness is trivial. Pick $E=\Omega$ and $Y=X\equiv 1$ to see that both sides are equal. 

From now on, we  assume that $1<p<2$. Consider the function 
\begin{equation}\label{gamma}
 \gamma(t)=\exp(pt^{p-1})\int_t^\infty \exp(-ps^{p-1})\mbox{d}s,\qquad t\geq 0.
 \end{equation}
%Let us list a few important properties of $\gamma$.

\begin{lemma}\label{PropU}
The function $\gamma$ has the following properties.

(i)  We have
$$\gamma(0)=p^{-1/(p-1)}\Gamma\left(\frac{p}{p-1}\right).$$

(ii) It satisfies the differential equation
\begin{equation}\label{differ}
1+\gamma'(t)=p(p-1)t^{p-2}\gamma(t).
\end{equation}

(iii) It is concave, nondecreasing and satisfies $\gamma'(t)\to 0$, as $t\to\infty$.
\end{lemma}
\begin{proof}
To compute $\gamma(0)$, simply substitute $r=ps^{p-1}$ under the integral. The condition (ii) follows from the direct differentiation. In view of \eqref{differ}, the concavity of $\gamma$ is equivalent to the estimate $ (p-2)\gamma(t)+\gamma'(t)<0.$ Applying  \eqref{differ} again, we rewrite the inequality in the form
\begin{equation}\label{rewr}
 \big(p(p-1)t^{p-1}+p-2\big)\gamma(t)\leq t.
\end{equation}
This is obvious if $p(p-1)t^{p-1}+p-2\leq 0$, so assume that the reverse estimate holds. 
Plugging the formula in \eqref{gamma} for $\gamma$, \eqref{rewr} can  be stated as
$$ F(t)=\frac{te^{-pt^{p-1}}}{p(p-1)t^{p-1}+p-2}-\int_t^\infty e^{-ps^{p-1}}\mbox{d}s\geq 0.$$
Now we compute that under our assumption that $1<p<2$, 
$$ F'(t)=\frac{(p-1)(p-2)e^{-pt^{p-1}}}{(p(p-1)t^{p-1}+p-2)^2}\leq 0.$$
It suffices to note that $F(t)\to 0$ when $t\to\infty$; thus, $F$ is nonnegative and $\gamma$ is concave. This automatically implies the remaining properties given in (iii): the first of them follows from the fact that $\gamma\geq 0$, while the convergence $\lim_{t\to\infty}\gamma'(t)=0$ is a consequence of 
\eqref{differ}.
\end{proof}

Next, let $H:[\gamma(0),\infty) \to [0,\infty)$ be the inverse to the function $t \mapsto t+\gamma(t)$. 
 To define the special function $U$ corresponding to \eqref{weak}, introduce the kernel
 $$ \alpha(\lambda)=\frac{1}{2}\gamma(H(\lambda))^{-2}\gamma'(H(\lambda))H'(\lambda)\lambda^2\chi_{[\gamma(0),\infty)}(\lambda)$$
and let
$$ U(x,y)=\int_{0}^\infty \alpha(\lambda) u_1(x/\lambda,y/\lambda) \mbox{d}\lambda+\frac{\gamma(0)}{2}.$$
Let us derive the explicit formula for $U$.

\begin{lemma}
We have
\begin{equation}\label{fU1} 
 U(x,y)= \frac{|y|^2-|x|^2}{2\gamma(0)}+\frac{\gamma(0)}{2}
\end{equation} if $|x|+|y|\leq \gamma(0)$, and
\begin{equation}\label{fU2}
\begin{split}
U(x,y)=& |y|-H(|x|+|y|)^p-pH(|x|+|y|)^{p-1}\big(|x|-H(|x|+|y|)\big),
\end{split}
\end{equation} if $|x|+|y|>\gamma(0)$.
\end{lemma}
\begin{proof} Of course, it suffices to prove the formula for $n=1$ and nonnegative  $x$, $y$. 
If $x+y\leq \gamma(0)$, then
$$ U(x,y)=\frac{y^2-x^2}{2}\int_{\gamma(0)}^\infty \left[-\frac{1}{\gamma(H(\lambda))}\right]'\mbox{d}\lambda+\frac{\gamma(0)}{2}=\frac{y^2-x^2}{2\gamma(0)}+\frac{\gamma(0)}{2}.$$
To prove \eqref{fU2} for $x+y>\gamma(0)$, it suffices to show that both sides have the same partial derivatives with respect to $y$. We have
$$ U(x,y)=(y^2-x^2)\int_{x+y}^\infty a(\lambda)\lambda^{-2}\mbox{d}\lambda+\int_{\gamma(0)}^{x+y}a(\lambda)(1-2x/\lambda)\mbox{d}\lambda+\frac{\gamma(0)}{2},$$
so
$$ U_y(x,y)=2y\int_{x+y}^\infty a(\lambda)\lambda^{-2}\mbox{d}\lambda=y\int_{x+y}^\infty \left[-\frac{1}{\gamma(H(\lambda))}\right]'\mbox{d}\lambda=\frac{y}{\gamma(H(x+y))}.$$
On the other hand, the $y$-derivative of the right-hand side of \eqref{fU2} equals
$$ 1+p(p-1)H(x+y)^{p-2}H'(x+y)(H(x+y)-x).$$
But, by the very definition of $H$ and $\gamma$, we have
$$ H'(x+y)=\frac{1}{1+\gamma'(H(x+y))}=\frac{1}{p(p-1)H(x+y)^{p-2}\gamma(H(x+y))},$$
so the derivative equals
$$ 1+\frac{H(x+y)-x}{\gamma(H(x+y))}=\frac{\gamma(H(x+y))+H(x+y)-x}{\gamma(H(x+y))}=\frac{y}{\gamma(H(x+y))}=U_y(x,y).$$
This completes the proof.
 \end{proof}
We turn to the majorization property.

\begin{lemma}\label{lemma_maj2}
For any $(x,y)\in\R^n\times \R^n$ we have
\begin{equation}\label{maj2}
U(x,y)\geq \max\left\{|y|,\frac{\gamma(0)}{2}\right\}-|x|^p.
\end{equation}
\end{lemma}
\begin{proof}
Again, we may assume that $n=1$ and $x,\,y\geq 0$. We split the reasoning into two parts.

\smallskip

\emph{The case $y\leq \gamma(0)/2$}. Arguing as above,  it suffices to show the majorization for $y=0$. If $x\leq \gamma(0)$, the inequality is equivalent to $ x^{2-p}\leq 2\gamma(0)$ and thus it is enough to check it for $x=\gamma(0)$. By Lemma \ref{PropU} (i), this is equivalent to
$$ \Gamma\left(\frac{p}{p-1}\right)^{p-1}\geq \frac{p}{2}.$$
This inequality is true, since the left-hand side is at least $1$, while the right-hand side does not exceed $1$. Now, assume that $x>\gamma(0)$ (and still, $y=0$). The inequality \eqref{maj2} reads
$$ -H(x)^p-pH(x)^{p-1}\big(x-H(x)\big)\geq \frac{\gamma(0)}{2}-x^p,$$
or, after the substitution $x=t+\gamma(t)$, $t\geq 0$,
\begin{equation}\label{intr}
 G(t):=(t+\gamma(t))^p-t^p-pt^{p-1}\gamma(t)\geq \frac{\gamma(0)}{2}.
\end{equation}
This is true for sufficiently large $t$; indeed, by the mean-value property, \eqref{differ} and Lemma \ref{PropU} (iii),
\begin{align*}
 G(t)&\geq {p(p-1)}(t+\gamma(t))^{p-2}\gamma(t)^2/2\\
 &=\frac{\gamma(t)}{2}\cdot p(p-1)t^{p-2}\gamma(t)\cdot \left(\frac{t+\gamma(t)}{t}\right)^{p-2}\\
 &>\frac{\gamma(0)}{2}\cdot \left(\frac{t+\gamma(t)}{t}\right)^{p-2}\geq \frac{\gamma(0)}{2},
\end{align*}
provided $t$ is large enough. Thus, if \eqref{intr} does not hold for all $t$, then there must be $t_0>\gamma(0)$ such that $G(t_0)<0$ and $G'(t_0)=0$. The latter equality is equivalent to
$$ (t_0+\gamma(t_0))^{p-1}-t_0^{p-1}=1/p,$$
and then, by \eqref{differ},
\begin{align*}
 G(t_0)&=(t_0+\gamma(t_0))(t_0^{p-1}+1/p)-t_0^p-pt_0^{p-1}\gamma(t_0)\\
 &=p^{-1}\left[t_0+\gamma(t_0)-p(p-1)t_0^{p-1}\gamma(t_0)\right]\\
 &=p^{-1}(\gamma(t_0)-t_0\gamma'(t_0)).
\end{align*}
It suffices to note that $\gamma(t_0)-t_0\gamma'(t_0)\geq \gamma(0)>0$, in view of the concavity of $\gamma$. This implies $G(t_0)>0$, a contradiction. This proves the majorization for $y\leq \gamma(0)/2$.

\smallskip

\emph{The case $y>\gamma(0)/2$}. This is much simpler. It suffices to focus on the majorization for $x+y\geq \gamma(0)$. Indeed, if the reverse inequality holds true, we rewrite \eqref{maj2} in the form
$$ \frac{y^2-x^2}{2\gamma(0)}+\frac{\gamma(0)}{2}-y+x^p\geq 0$$
and note that the left hand side decreases as $y$ increases. If $x+y\geq \gamma(0)$, the majorization reads
$$ x^p-H(x+y)^p\geq pH(x+y)^{p-1}(x-H(x+y)),$$
which follows immediately from the mean-value property. In particular, let us observe here that if $y=\gamma(x)$, then both sides of \eqref{maj2} are equal (then $x=H(x+y)$). This will be important for us later, in the proof of the sharpness.
\end{proof}

\begin{proof}[Proof of \eqref{weak}]
It suffices to show the assertion under the assumption $||X||_p<\infty$, since otherwise the bound is obvious. By Lemma \ref{mainmar}, the formula for $U$ and Fubini's theorem, we obtain
 $ \E U(X_t,Y_t)\leq {\gamma(0)}/{2}$ for all $t\geq 0$. Combining this with \eqref{maj2} yields
$ \E \max\{|Y_t|,\gamma(0)/2\}\leq \E |X_t|^p+\gamma(0)/2,$ or
$$ \E \max\left\{|Y_t|-\frac{\gamma(0)}{2},0\right\}\leq \E |X_t|^p.$$
Now pick an arbitrary event $E\in \F$ and consider its splitting into the sets
$$ E^-=E\cap\{|Y_t|\leq \gamma(0)/2\},\qquad E^+=E\cap\{|Y_t|>\gamma(0)/2\}.$$
We have $\E |Y_t|1_{E^-}\leq \mathbb{P}(E^-)\cdot\gamma(0)/2$ and
$$ \E \left\{|Y_t|-\frac{\gamma(0)}{2}\right\}1_{E^+}\leq \E \max\left\{|Y_t|-\frac{\gamma(0)}{2},0\right\}\leq \E |X_t|^p.$$
Adding the last two inequalities yields 
$$ \E |Y_t|1_E\leq ||X||_p^p+\gamma(0)\mathbb{P}(E)/2.$$
Now fix $\lambda>0$ and apply this estimate to a new martingale pair $X/\lambda$, $Y/\lambda$. Clearly, the differential subordination is preserved, so the use of the bound is permitted and we obtain
$$ \E |Y_t|1_E\leq \lambda^{1-p}||X||_p^p+\lambda\gamma(0)\mathbb{P}(E)/2.$$
A straightforward analysis shows that as a function of $\lambda$, the right hand side attains its minimum for
$$ \lambda=\left(\frac{2(p-1)||X||_p^p}{\gamma(0)\mathbb{P}(E)}\right)^{1/p},$$
and, plugging the formula for $\gamma(0)$ (see Lemma \ref{PropU} (i)), we obtain the bound
$$ \E |Y_t|1_E\leq \left(\frac{1}{2}\Gamma\left(\frac{2p-1}{p-1}\right)\right)^{1-1/p}||X||_p\mathbb{P}(E)^{1-1/p}.$$
Taking the supremum over $t\geq 0$ completes the proof.
\end{proof}

\begin{proof}[Sharpness] The reasoning is similar to that concerning the logarithmic bound. We will construct an example for which both sides of \eqref{weak} are equal with $E=\Omega$. Let $B$ be a standard Brownian motion starting from $\gamma(0)/2$ and stopped at the exit time from $[0,\infty)$, and let 
$$ D_t=\frac{\gamma(0)}{2}+\int_0^t \operatorname*{sgn}B_s\;\mbox{d}B_s.$$
We easily check that for all $t\geq 0$ we have $|B_t|+|D_t|\geq \gamma(0)$. Introduce the stopping time $ \tau=\inf\{t\geq 0: |D_t|=\gamma(B_t)\}$; it is easy to check that $\tau\in L^{p/2}$ for some $p>1$ (actually, one can show that $\tau\in L^{p/2}$ for all $p<\infty$, but we will not need this). Consider the martingales $ X=(B_{\tau\wedge t})_{t\geq 0}$, $Y=(D_{\tau\wedge t})_{t\geq 0}.$  
Since $U$ is of class $C^2$ on the set $\{(x,y):x\geq 0,x+y\geq \gamma(0)\}$ and satisfies 
$$U_{xx}(x,y)+2U_{xy}(x,y)\cdot \operatorname*{sgn}y+U_{yy}(x,y)=0\quad \mbox{and}\quad U_y(x,0)=0$$
on this set, a combination of It\^o-Tanaka formula and a limiting argument yields
$$ \E U(X_\infty,Y_\infty)=\E U(X_0,Y_0)=\frac{\gamma(0)}{2}.$$
However, we have $U(x,\pm\gamma(x))=\gamma(x)-x^p$ for $x\geq 0$: see the end of the proof of Lemma \ref{lemma_maj2}. Since $|Y_\infty|=\gamma(X_\infty)$ almost surely, we obtain
$$ \E |Y_\infty|=\E |X_\infty|^p+\frac{\gamma(0)}{2}.$$
Thus, by Young's inequality,
\begin{align*}
 C_p||X||_p&=p^{1/p}||X_\infty||_p\cdot \left(\frac{p}{p-1}\cdot\frac{\gamma(0)}{2}\right)^{1-1/p}\\
 &\leq ||X_\infty||_p^p+\frac{\gamma(0)}{2}=\E |Y_\infty|
\end{align*}
and both sides of \eqref{weak} must be equal. This completes the proof.
\end{proof}

\subsection{Proof of Theorem \ref{weak_thm}--case  $2<p<\infty$} This time the reasoning is much simpler. The special function is given by the formula
$$ U(x,y)=\frac{p^p(p-1)^{2-p}(p-2)}{4}\int_0^{1-p^{-1}} \lambda^{p-1}u_\infty(x/\lambda,y/\lambda)\mbox{d}\lambda,$$
We easily compute that
$$ U(x,y)=\frac{1}{2}\left(\frac{p}{p-1}\right)^{p-1}(|y|-(p-1)|x|)(|x|+|y|)^{p-1},$$
if $|x|+|y|\leq 1-p^{-1}$, while for remaining $(x,y)$,
$$ U(x,y)=\frac{p^2}{4}\left[|y|^2-|x|^2-\frac{2(p-2)|y|}{p}+\frac{(p-1)^2(p-2)}{p^3}\right].$$ 
We have the following majorization.

\begin{lemma}\label{maj_lemma3}
For any $(x,y)\in \R^n\times \R^n$ we have 
\begin{equation}\label{maj3}
U(x,y)\geq p\max\left\{|y|-1+\frac{1}{p},0\right\}-\frac{p^{p-1}}{2}|x|^p.
\end{equation}
\end{lemma}
\begin{proof}
As previously, we may assume that $n=1$ and $x,\,y\geq 0$. If $x+y<1-1/p$, then the inequality is equivalent to
$$ (y-(p-1)x)(x+y)^{p-1}+(p-1)^{p-1}x^p\geq 0.$$
But this is true for all nonnegative $x,\,y$. A straightforward analysis of the derivative shows that for a fixed $x$, the left hand side (considered as a function of $y$) attains its minimum for $y=(p-2)x$; this minimum is $0$. Next, suppose that $y\geq 1-1/p$ and put all the terms of \eqref{maj3} on the left-hand side. Then, for a fixed $x$, the expression on the left is a quadratic function of $y$ which attains its minimum for $y=1$. However, for this value of $y$, the majorization is equivalent to
\begin{equation}\label{meanvalue}
 (px)^p-1\geq \frac{p}{2}((px)^2-1),
\end{equation}
which follows immediately from the mean-value property. Finally, if $x+y>1-1/p>y$, \eqref{maj3} becomes
$$ \frac{p^2}{4}\left[y^2-x^2-\frac{2(p-2)y}{p}+\frac{(p-1)^2(p-2)}{p^3}\right]\geq -\frac{p^{p-1}x^p}{2}.$$
But this bound holds true for all $x,\,y$. Indeed, observe that as a function of $y$, the left-hand side attains its minimum for $y=1-2/p$, and for this choice of $y$, the inequality again reduces to \eqref{meanvalue}.
\end{proof}

\begin{proof}[Proof of Theorem \ref{weak_thm}] By Lemma \ref{mainmar}, the definition of $U$ and \eqref{maj3}, we obtain
$$ \E \left\{|Y_t|-1+\frac{1}{p},0\right\}\leq \frac{p^{p-2}}{2}\E |X_t|^p.$$
Arguing as previously, this leads to the bound
$$ \E |Y_t|1_E\leq \frac{p^{p-2}}{2}\E |X_t|^p+\left(1-\frac{1}{p}\right)\mathbb{P}(E).$$
Apply this inequality to the martingales $X/\lambda$, $Y/\lambda$, multiply both sides by $\lambda$ and optimize the right-hand side over $\lambda$. It turns out that the choice
$$ \lambda=\left(\frac{p^{p-1}\E |X_t|^p}{2\mathbb{P}(E)}\right)^{1/p}$$
makes the right-hand side minimal and we obtain
$$ \E |Y_t|1_E\leq  \left(\frac{p^{p-1}}{2}\right)^{1/p}||X_t||_p\mathbb{P}(E)^{1-1/p}.$$
This yields \eqref{weak}, by taking the supremum over all $t$. To prove that this estimate is sharp, pick an arbitrary pair $(X,Y)$ of real-valued martingales such that $Y$ is differentially subordinate to $X$. Introduce the stopping time $\tau=\inf\{t\geq 0:|Y_t|\geq 1\}$. Then the stopped martingale $Y^\tau$ is differentially subordinate to $X$ and thus, applying \eqref{weak} with $E=\{\sup_{s\geq 0}|Y_s|\geq 1\}$, we get
$$ \sup_{t\geq 0}\E |Y_{\tau\wedge t}|1_{\{\sup_s|Y_s|\geq 1\}}\leq \frac{p^{p-1}}{2}||X||_p^p,$$
In turn, this inequality implies
$$ \mathbb{P}\left(\sup_{s\geq 0}|Y_s|> 1\right)\leq  \frac{p^{p-1}}{2}||X||_p^p,$$
which is sharp, as proved by Suh \cite{Suh}.
\end{proof}

%%%%%%%%%%%%%%%%%%%%%%%%%%%%%%%%%%%%%%%%%%%%%%%%%%%%%%%%%%%%%%%%%%%%%%%%
%%%%%%%%%%%%%%%%%%%%%%%%%%%%%%%%%%%%%%%%%%%%%%%%%%%%%%%%%%%%%%%%%%%%%%%%
%%%%%%%%%%%%%%%%%%%%%%%%%%%%%%%%%%%%%%%%%%%%%%%%%%%%%%%%%%%%%%%%%%%%%%%%

\section{Applications} 

\subsection{Logarithmic and weak-type bounds  for martingale transforms on manifolds}\label{manifoldtransforms}
In this section we will apply the probabilistic results (which we have just established) in the study of Riesz transforms on Lie groups. We start from the brief description of the connection between these two environments, and for the detailed study of the interplay we refer the interested reader to \cite{IW}. Suppose that $M$ is an $n$-dimensional Riemannian manifold with Ricci curvature bounded from below (this additional assumption guarantees that the Brownian motion on $M$ does not explode in finite time, see Emery \cite{Em}). Let $\langle\cdot,\cdot\rangle$ denote the inner product on $TM$, the tangent space to $M$. A Brownian motion in $M$ is an $(\F_t)_{t\geq 0}$ adapted process $(B_t)_{t\geq 0}$ with values in $M$ such that for all smooth functions $f:M\to \R$, the process
\begin{equation}\label{itod}
 I_{df}=\left(f(B_t)-f(B_0)-\frac{1}{2}\int_{0+}^t \Delta_M f(B_s)\mbox{d}s\right)_{t\geq 0}
 \end{equation}
is a real-valued continuous martingale. See the monograph \cite{Em} for more on the subject.

Next, let $\mathfrak{K}$ be a continuous, adapted process with values in $T^*M$, the cotangent space of $M$. We say that $\mathfrak{K}$ is \emph{above} $B$, if for all $t\geq 0$ and $\omega\in \Omega$ we have $\mathfrak{K}_t(\omega)\in T^*_{B_t(\omega)}M$. Having assumed this, we can define $I_\mathfrak{K}=\left(\int_0^t \langle \mathfrak{K}_s,\mbox{d}B_s\rangle\right)_{t\geq 0}$, the It\^o integral of $\mathfrak{K}$ with respect to $B$, by requiring that
\smallskip
\begin{itemize}
\item[(i)] if $\mathfrak{K}_t=\mbox{d}f(B_t)$ for some smooth function $f:M\to \R$, then $I_\mathfrak{K}$ equals $I_{df}$ given by \eqref{itod}.
\item[(ii)] if $K$ is a real valued, continuous process, then $I_{K\mathfrak{K}}=\left(\int_0^t K_s\mbox{d}(I_\mathfrak{K})_s\right)_{t\geq 0}$ is the classical It\^o integral of $K$ with respect to the continuous martingale $I_\mathfrak{K}$.
\end{itemize}
\smallskip
These two conditions determine uniquely the class of stochastic integrals. It can be easily verified that if $\mathfrak{K}$ is above $B$, then the process $I_\mathfrak{K}$ is a continuous, real-valued martingale. The covariance process of two such integrals can be expressed by the formula
\begin{equation}\label{covariance}
[I_\mathfrak{K},I_\mathfrak{L}]_t=\int_0^t \operatorname*{Trace}(\mathfrak{K}_s\otimes \mathfrak{L}_s)\mbox{d}s,
\end{equation}
where $\otimes$ is the tensor product and $(\mathfrak{K}_s\otimes \mathfrak{L}_s)(\omega)=\mathfrak{K}_s(\omega)\otimes\mathfrak{L}_s(\omega)\in T^*_{B_s(\omega)}\otimes T^*_{B_s(\omega)}$.

Now assume that $x\in M$ and let $End(T_x^*M)$ be the space of all linear maps from $T^*_xM$ to itself. Let $End(T^*M)$ be the collection of all $End(T^*_xM)$, $x\in M$.  A bounded and continuous process $A$ with values in $End(T^*_xM)$ is called \emph{a martingale transformer with respect to $B$}, if for all $t\geq 0$ and $\omega\in \Omega$ we have $A_t(\omega)\in End(T_{B_t(\omega)}^*M)$. Such an object induces an important action on the class of stochastic integrals. Namely, suppose that $\mathfrak{K}$ is a continuous, bounded process with values in $T^*M$ which is above $B$, and let $A$ be a martingale transformer with respect to $B$. Then $A*I_\mathfrak{K}$, \emph{the martingale transform of $I_\mathfrak{K}$ by $A$}, is the real-valued martingale defined by the identity
$$ A*I_\mathfrak{K}=I_{A\mathfrak{K}}=\left(\int_0^t \langle A_s\mathfrak{K}_s,\mbox{d}B_s\rangle\right)_{t\geq 0}.$$
In the particular case when $\mathfrak{K}=\mbox{d}f$ for some smooth function $f:M\to \R$, we will use the notation $A*f$ instead of $A*I_{df}$. Given a sequence $\mathcal{A}=(A_1,A_2,\ldots,A_d)$ of martingale transformers above $B$, we define $\mathcal{A}*I_{\mathfrak{K}}$ as the $d$-dimensional martingale $(A_1*I_{\mathfrak{K}},A_2*I_{\mathfrak{K}},\ldots,A_d*I_{\mathfrak{K}})$. We introduce the norm of $\mathcal{A}$ by
$$ |||A|||=\sup\left(\sum_{j=1}^d |A_{j,t}(\omega)e|^2\right)^{1/2},$$
where the supremum is taken over all $\omega\in \Omega$, all $t\geq 0$ and all vectors $e\in T_{B_t(\omega)}M$ of length $1$. If $A$ is a single martingale transformer, then we define its norm by $|||A|||=|||(A)|||$.

Theorems studied in the preceding section lead to the following estimates for martingale transforms on manifolds.

\begin{theorem}\label{stoch}
Let $\mathfrak{K}$ be a bounded, continuous, $T^*M$-valued process above $B$.

(i) If $\mathcal{A}$ be a martingale transformer above $B$, then for any  $E\in \F$ we have
\begin{equation}\label{vpo1}
\sup_{t\geq 0}\E |(\mathcal{A}*I_{\mathfrak{K}})_t|1_E\leq K\sup_{t\geq 0}\E \Psi\big(|||\mathcal{A}|||\;|(I_\mathfrak{K})_t|\big)+L(K)\mathbb{P}(E),\qquad K>1,
\end{equation}
and 
\begin{equation}\label{vpo2}
\quad \sup_{t\geq 0}\E |(\mathcal{A}*I_{\mathfrak{K}})_t|1_E\leq K_p|||\mathcal{A}|||\;||I_{\mathfrak{K}}||_p\mathbb{P}(E)^{1-1/p},\qquad \qquad  1<p<\infty.
\end{equation}

(ii) If $A$ is a single martingale transformer satisfying the condition $\langle A_t(\omega)\xi,\xi\rangle=0$ for all $t\geq 0$, $\omega\in \Omega$ and $\xi\in T^*_{B_t(\omega)}M$, then
\begin{equation}\label{po2}
\sup_{t\geq 0}\E \Phi\left(|(A*I_\mathfrak{K})_t|/K\right)\leq \frac{L(K)|||A|||\;||X||_1}{K},\qquad K>2/\pi,
\end{equation}
and
\begin{equation}\label{po3}
||A*I_\mathfrak{K}||_q\leq C_p|||A|||\;||I_\mathfrak{K}||_1^{1/q}||I_\mathfrak{K}||_\infty^{1/p},\qquad 1<p<\infty.
\end{equation}
\end{theorem}
\begin{proof}
The assertion will follow immediately from the results of Section 2, once we have proven that the appropriate martingales satisfy differential subordination and orthogonality. To show this, pick $t\geq 0$, $\omega\in \Omega$ and let $x=B_t(\omega)\in M$. Let $e_1$, $e_2$, $\ldots$, $e_n$ be an orthonormal basis for $T_xM$, the tangent space to $M$ at $x$. Then for each $j\in \{1,\,2,\,\ldots,\,d\}$,
\begin{align*}
\operatorname*{Trace}\big(A_j\mathfrak{K}_t(\omega)\otimes A_j\mathfrak{K}_t(\omega)\big)&=\sum_{k=1}^n \big(A_j\mathfrak{K}_t(\omega)\otimes A_j\mathfrak{K}_t(\omega)\big)(e_k,e_k)\\
&=\sum_{k=1}^n \big|<A_j\mathfrak{K}_t(\omega),e_k>\big|^2=|A_j\mathfrak{K}_t(\omega)|^2,
\end{align*}
where $<\cdot,\cdot>:T_x^*M\times T_xM\to \R$ stands for the duality product. Therefore, by \eqref{covariance}, for any $0\leq s\leq t$ we may write
\begin{align*}
[\mathcal{A}*I_\mathfrak{K},\mathcal{A}*I_\mathfrak{K}]_t-[\mathcal{A}*I_\mathfrak{K},\mathcal{A}*I_\mathfrak{K}]_s&=\sum_{j=1}^d\int_{s+}^t  \operatorname*{Trace}(A_j\mathfrak{K}_u\otimes A_j\mathfrak{K}_u)\mbox{d}u\\
&=\sum_{j=1}^d\int_{s+}^t |A_j\mathfrak{K}_u|^2\mbox{d}u\\
&\leq |||\mathcal{A}|||^2\int_{s+}^t|\mathfrak{K}_u|^2\mbox{d}u\\
&=[|||\mathcal{A}|||I_\mathfrak{K},|||\mathcal{A}|||I_\mathfrak{K}]_t-[|||\mathcal{A}|||I_\mathfrak{K},|||\mathcal{A}|||I_\mathfrak{K}]_s,
\end{align*}
which is the desired differential subordination. The proof of the orthogonality goes along the same lines: one shows that $\operatorname*{Trace}(A\mathfrak{K}_t(\omega)\otimes \mathfrak{K}_t(\omega))=0$ for all $t$, $\omega$ and obtains $\mbox{d}[A*I_\mathfrak{K},|||A|||I_\mathfrak{K}]=0$ directly from \eqref{covariance}.
\end{proof}

\subsection{Logarithmic and weak-type inequalities for Riesz transforms on Lie groups}\label{LieRiesz}
Now we will describe an elegant  probabilistic representation of first order Riesz transforms on Lie groups $G$ in terms of martingale transforms with respect to the Brownian motion with values in $G\times \R$. The construction goes back to the classical  paper \cite{GV} of Gundy and Varopoulos, in which the case $G=\R^n$ was studied. The idea has been generalized in several directions and exploited in many papers; see e.g. \cite{Ar, BW, G1, G2, V}.

To this end, assume that $G$ is a compact connected Lie group  of dimension $n$,  endowed with a Riemannian bi-invariant metric and let $\mbox{d}x$ denote the usual Riemannian volume  measure on $G$. Suppose that $\mathfrak{g}$ denote the Lie algebra of $G$ and let $\{X_1,\,X_2,\,\ldots,\,X_n\}$ be an orthonormal basis for $\mathfrak{g}$. Consider the group $\tilde{G}=G\times \R$, with the product Riemannian metric and the corresponding Lie algebra $\mathfrak{g}\oplus \R$. If $X_0=\partial/\partial y$ is the generator of the Lie algebra of $\R$, then $\{X_1,\,X_2,\,\ldots,\,X_n,\,X_0\}$ forms an orthonormal basis of $\mathfrak{g}\oplus \R$. 

Let $X$, $Y$ be independent Brownian motions in $G$ and $\R$, respectively; then $Z=(X,Y)$ is a Brownian motion in the product group $\tilde{G}$. Fix $\lambda>0$ and assume that the initial distribution of $Z^\lambda=(Z_t)_{t\geq 0}$ is the product measure $\mbox{d}x\times \delta_{\lambda}$, where $\delta_\lambda$ is the Dirac measure concentrated on $\{\lambda\}$. Put $\tilde{G}^+=G\times [0,\infty)$ and introduce the stopping time 
$$\tau_0=\inf\{t\geq 0:Y_t\leq 0\}.$$
 Then $(Z^{\lambda}_{\tau_0\wedge t})_{t\geq 0}$ is a Brownian motion in $\tilde{G}^+$, stopped at the boundary of this set. Let $A:\tilde{G}^+\to End(T^*\tilde{G}^+)$ be an arbitrary  continuous section of the bundle $End(T^*\tilde{G}^+)$, and consider the process $\tilde{A}=\big(A(Z_{\tau_0\wedge t})\big)_{t\geq 0}$. Then $\tilde{A}$ is a martingale transformer. 
Fix a function $f\in C_0^\infty(G)$ and let $F$ be its Poisson extension to $\tilde{G}^+$. That is,  the unique $C^\infty$ function on $\tilde{G}$ satisfying
$$ 0=\Delta_{\tilde{G}}F(x,y)=\Delta_GF(x,y)+\frac{\partial^2 F}{\partial y^2}(x,y),\qquad x\in G,\,y>0,$$
and such that $F$ is bounded on $\tilde{G}$ (see  \cite{G2} and \cite{St0}). 
Now, for $A$, $f$, $F$ and $\lambda$ as above, define the projection of the $A$-transform of $f$ by
$$ T_A^\lambda f(x)=\E\big[\tilde{A}*dF|Z_{\tau_0}=x\big],$$
the conditional expectation of $\tilde{A}*dF$ with respect to the $\sigma$-algebra generated by $Z_{\tau_0}$. Since $Z_{\tau_0}$ takes values in the boundary of $\tilde{G}\times \{0\}$,  $T_A^\lambda f$ can be interpreted as a function on $G$. 

Recall that $\{X_1,\,X_2,\,\ldots,\,X_n,\,X_0\}$ is an orthonormal basis of $\mathfrak{g}\oplus \R$. For a given $j\in\{1,\,2,\,\ldots,\,n\}$, let $R_j=R_{X_j}=X_j\circ(-\Delta_G)^{-1/2}$ be the Riesz transform on $G$ 
in the direction $X_j$. These operators are defined by the following requirement: for any $f:G\to \R$, we have
$$ R^Gf(a)=\sum_{j=1}^n R_jf(a)X_j(a)\qquad \mbox{ for all }a\in G,$$
where $X_j(a)$ is the vector field $X_j$ evaluated at the point $a$. 
Consider the linear maps $A^j,\,E^j:\mathfrak{g}\oplus \R\to\mathfrak{g}\oplus \R$ given by
$$ A^jX_m=\begin{cases}
X_j & \mbox{if }m=0,\\
-X_0 & \mbox{if }m=j,\\
0 & \mbox{otherwise},
\end{cases}
\qquad  E^jX_m=\begin{cases}
X_j & \mbox{if }m=0,\\
0 & \mbox{otherwise}.
\end{cases}$$

Clearly, $A^j$ defines a smooth section of $End(T\tilde{G}^+)$ and can be regarded as a martingale transformer with the use of the natural identification between $\mathfrak{g}\oplus \R$ and its dual, induced by the Riemannian metric. We have the following statement, which follows immediately from the results of Arcozzi \cite{Ar}.

\begin{theorem}\label{Arcozzi}
Let $f\in C_0^\infty(G)$. Then 
$$\lim_{\lambda\to\infty}T^\lambda_{A^j}f = R_jf\qquad \mbox{in }L^p(G), \quad 1\leq p<\infty,$$
and
$$\lim_{\lambda\to\infty}T^\lambda_{E^j}f = -\frac{1}{2}R_jf\qquad \mbox{in }L^p(G), \quad 1\leq p<\infty.$$
\end{theorem}

\begin{remark}  Using the space time Brownian motion construction introduced in \cite{BMH} and the Fourier transform  (Peter-Weyl), a quite direct and simple probabilistic representation for second order Riesz transforms is given in \cite{AppBan} and \cite{BanBau}.  Following that  argument with the space time Brownian motion replaced by the Brownian motion $Z$  above leads to a slightly different construction of first order Riesz transforms on $G$.  We leave the details to the interested reader. 
\end{remark}

With our probabilistic representation for Riesz transforms on $G$, we are ready to establish their logarithmic and weak-type inequalities.  In what follows we again use  $|E|=\int_G \chi_A \mbox{d}x$ to denote  the volume measure of $E\subset G$. Recall the constant $K_p$ given by \eqref{defKp}.

\begin{theorem}\label{vR}
(i) For any $K>2/\pi$, any $f:G\to \R$ with $\int_G\Psi(|f|)<\infty$ and any Borel subset $E$ of $G$ we have
\begin{equation}\label{vpg1}
\int_E |R^Gf(x)|\mbox{d}x\leq 2K\int_G\Psi(|f(x)|)\mbox{d}x+ \frac{|E|}{K-1}.
\end{equation}

(ii) For any $1<p<\infty$, any $f\in L^p(G)$ and any Borel subset $E$ of $G$ we have
\begin{equation}\label{vpg2}
\int_E |R^Gf(x)|\mbox{d}x\leq 2K_p||f||_{L^p(G)}|E|^{1-1/p}.
\end{equation}
\end{theorem}
\begin{proof}
We will only establish (i), the reasoning leading to (ii) is analogous. By standard density arguments, it suffices to prove the bound for $f\in C^\infty(G)$. Consider the martingale transformer $\mathcal{A}=(E^1,E^2,\ldots,E^n)$; directly from the definition, we derive that $|||A|||=1$. Now, recall the inequality \eqref{intin} established in the proof of Theorem \ref{log_thm}. Letting $t\to\infty$, we see that this intermediate bound leads to the estimate
\begin{equation}\label{inr}
 \E \max\left\{|(\mathcal{A}*I_{dF})_\infty|-\frac{1}{2(K-1)},0\right\}\leq K\E \Psi( |(I_{dF})_\infty|),
\end{equation}
see the proof of Theorem \ref{stoch}. The function $x\mapsto \max\{|x|-1/(2(K-1)),0\}$ is convex, so 
\begin{align*}
&\int_G \max\left\{\frac{1}{2}|R^Gf(x)|-\frac{1}{2(K-1)},0\right\}\mbox{d}x\\
 &\qquad \leq \liminf_{\lambda\to \infty}\; \int_G \max\left\{|T^\lambda_{\mathcal{A}}f(x)|-\frac{1}{2(K-1)},0\right\}\mbox{d}x\\
 &\qquad = \liminf_{\lambda\to \infty}\; \E \max\left\{|T^\lambda_{\mathcal{A}}f(B_{\tau_0})|-\frac{1}{2(K-1)},0\right\}\\
 &\qquad \leq \liminf_{\lambda\to \infty}\; \E \max\left\{|\tilde{\mathcal{A}}*I_{dF}|_\infty-\frac{1}{2(K-1)},0\right\}\\
&\qquad \leq K\E \Psi(|(I_{dF})_\infty|)\\
&\qquad =K\int_G\Psi(|f(x)|)\mbox{d}x.
 \end{align*}
Here in the first inequality we have used Fatou's lemma and Lemma \ref{Arcozzi}, then we have exploited conditional version of Jensen's inequality and finally we applied \eqref{inr}. 
Now we adapt the reasoning from the proof of Theorem \ref{log_thm}. If $E$ is an arbitrary subset of $G$, we split it into 
$$ E^-=E\cap\{|R^Gf(x)|<1/(K-1)\},\qquad E^+=E\cap\{|R^Gf(x)|\geq 1/(K-1)\},$$
and write
\begin{align*} \int_{E^-}|R^Gf(x)|\mbox{d}x&\leq \frac{|E^-|}{K-1},\\
\int_{E^+}|R^Gf(x)|\mbox{d}x-\frac{|E^+|}{K-1}&\leq \int_G \max\left\{|R^Gf(x)|-\frac{1}{K-1},0\right\}\mbox{d}x\\
&\leq 2K\int_G\Psi(2|f(x)|)\mbox{d}x.
\end{align*}
It suffices to add the last two inequalities to get the claim.
\end{proof}

To prove related estimates for directional Riesz transforms, one requires an additional duality argument. Namely, first we show the following auxiliary bounds.
\begin{theorem}\label{group}
Let $j\in \{1,\,2,\,\ldots,\,d\}$ and $f\in L^\infty(G)$ be fixed. 

(i) If $K>2/\pi$ and $||f||_{L^\infty(G)}\leq 1$, then
\begin{equation}\label{pg1}
 \int_G \Phi\left(|R_jf(x)|/K\right)\mbox{d}x\leq \frac{L(K)||f||_{L^1(G)}}{K}.
\end{equation}

(ii) For any $1<q<\infty$ we have
\begin{equation}\label{pg2}
||R_jf||_{L^q(G)}\leq C_p||f||_{L^1(G)}^{1/q}||f||_{L^\infty(G)}^{1/p}.
\end{equation}
\end{theorem}
\begin{proof}
The proof is similar to that of Theorem \ref{vR}. Namely, one exploits the one-dimensional martingale transformer $A^j$, which satisfies $|||A|||=1$ and $\langle A\xi,\xi\rangle=0$ for all $\xi$. The further details are omitted and left to the reader.
\end{proof}

Now we are ready to deduce the logarithmic and weak-type estimates for directional Riesz transforms.
\begin{theorem}\label{group1}
Let $j\in \{1,\,2,\,\ldots,\,n\}$.

(i) For any $K>2/\pi$, any $f:G\to \R$ with $\int_G\Psi(|f|)<\infty$ and any Borel subset $E$ of $G$ we have
\begin{equation}\label{logriesz}
\int_E |R_jf(x)|\mbox{d}x\leq K\int_G\Psi(|f(x)|)\mbox{d}\chi(x)+L(K)\cdot |E|.
\end{equation}

(ii) For any $1<p<\infty$, any $f\in L^p(G)$ and any subset $E$ of $G$ we have
\begin{equation}\label{weakriesz}
 \int_E|R_jf(x)|\mbox{d}x\leq C_p ||f||_{L^p(G)}|E|^{1-1/p}.
\end{equation}
\end{theorem}
\begin{proof}
Consider the decomposition of $L^2(G)=\bigoplus_{k=1}^\infty\mathcal{H}_k$ into eigenspaces for $\Delta_G$, provided by Peter-Weyl theorem \cite{St0}. Thus, $\mathcal{H}_k\subset C_0^\infty(G)$ and $\Delta_Gf=-\mu_kf$ for $f\in\mathcal{H}_k$, where $0<\mu_1<\mu_2<\ldots$ is the sequence of of eigenvalues of $-\Delta_G$. 
Fix $f=\sum_{k=1}^N f_k$, with $f_k\in\mathcal{H}_k$, $k=1,\,2,\,\ldots,\,N$, and put
$g=1_ER_jf/|R_jf|$ ($g=0$ if the denominator is zero). Let $g=\sum_{k=1}^\infty g_k$ be the decomposition of $g$, with $g_k\in \mathcal{H}_k$ for each $k$. If $k$, $m$ are different positive integers, then $\int_G (R_jf_k)g_m=0$ and hence, integrating by parts,
\begin{equation}\label{cchain}
\begin{split}
 \int_{E}|R_jf(x)|\mbox{d}x&= \int_{G} R_jf(x)\, g(x)\,\mbox{d}x\\
 &=\sum_{k=1}^N \int_{G}R_jf_k(x)\, g_k(x)\,\mbox{d}x\\
&=-\sum_{k=1}^N \int_{G}f_k(x)\, R_jg_k(x)\,\mbox{d}x\\
&=-\int_{G}f(x)R_jg(x)\,\mbox{d}x.
\end{split}
\end{equation}
Now, to prove (i), we bound the latter expression with the use of Young's inequality: it does not exceed
\begin{align*}
K\!\int_{G}\Psi(|f(x)|)\,\mbox{d}x+K\int_{G}\Phi(|R_jg(x)|/K)\,\mbox{d}x\leq  K\int_{G}\Psi(|f(x)|)\,\mbox{d}x+L(K)||g||_{L^1(G)}.
\end{align*}
Here in the last passage we have used \eqref{pg1} and the fact that $g$ takes values in $[-1,1]$. It suffices to note that $||g||_{L^1(G)}\leq |E|$ and use a standard density argument to obtain \eqref{logriesz} for arbitrary $f$. To prove (ii), we use H\"older inequality and \eqref{pg2} to bound the expression \eqref{cchain} from above by $ ||f||_{L^p(G)}||R_jg||_{L^q(G)}\leq C_p||f||_{L^p(G)}||g||^{1/q}_{L^1(G)}\leq C_p||f||_{L^p(G)} |E|^{1/q}$, which is  \eqref{weakriesz}.
\end{proof}

\subsection{Logarithmic and weak-type inequalities for Riesz transforms on spheres}\label{SphereRiesz}
The purpose of this section is to analyze the behavior of Riesz transforms on the unit sphere $\mathbb{S}^{n-1}=\{x\in \R^n:|x|=1\}$ equipped with the standard Riemannian metric and normalized $SO(n)$ invariant measure. The case $n=2$ is classical and well understood, so from now on we assume that $n\geq 3$.  We have that $\mathbb{S}^{n-1}$ is a Lie group only for $n=3$, so in general the methodology developed in the previous section does not apply and we need a new approach. 

Actually, we will work with two non-equivalent notions of Riesz transforms on the spheres (see e.g. Arcozzi and Li \cite{AX} for an overview of various types of Riesz transforms on $\mathbb{S}^{n-1}$). Both these transforms have been studied quite intensively in the literature. The two possibilities arise from the fact that there are two natural ways to ``fill in''  $\mathbb{S}^{n-1}$ so that it is the boundary of an $n$-dimensional Riemannian manifold. Let us analyze these separately.
 
Firstly, one can express $\mathbb{S}^{n-1}$ as the boundary of the cylinder $\mathbb{S}^{n-1}\times \R$, and this leads to the $R^{\mathbb{S}^{n-1}}$ already introduced at the beginning of the paper. For a fixed $1\leq \ell<m\leq n$, consider the differential operator
\begin{equation}\label{defT}
 \mathcal{T}_{\ell m}=x_\ell\partial_m-x_m\partial_\ell.
\end{equation}
If $x_\ell+ix_m=re^{i\theta}$, then $\mathcal{T}_m=\partial/\partial \theta$ is the derivative with respect to the angular coordinate in the $(x_\ell,x_m)$ plane and hence is a well defined vector field on $\mathbb{S}^{n-1}$. There is a useful formula which relates these vector fields to the spherical gradient $\nabla_{\mathbb{S}^{n-1}}$. Namely, if $f:\mathbb{S}^{n-1}\to \R$ is a smooth function, then
\begin{equation}\label{grad}
 |\nabla_{\mathbb{S}^{n-1}}f|=\left(\sum_{\ell<m}|\mathcal{T}_{\ell m}f|^2\right)^{1/2}.
\end{equation}
We define the directional Riesz transform (of cylinder type) by
$$ Q^c_{\ell m}=\mathcal{T}_{\ell m}\circ (-\Delta_{\mathbb{S}^{n-1}})^{-1/2}$$
and an auxiliary cylindrical Riesz transform $Q^c$ as the vector $(Q^c_{\ell m})_{1\leq \ell<m\leq n}$. Note that by \eqref{grad}, we have
$$ |R^{\mathbb{S}^{n-1}}|=|Q^c|,$$
so the analysis of $R^{\mathbb{S}^{n-1}}$ reduces to that of $Q^c$.

We turn to the second type of Riesz transform on $\mathbb{S}^{n-1}$ (cf. Kor\'anyi and V\'agi \cite{KV1, KV2}). Let $\mathcal{H}_k$ denote the space of spherical harmonics of degree $k$ and let
$$ \mathcal{E}_0=\left\{ f:\mathbb{S}^{n-1}\to\R\;:\;f=\sum_{k=1}^N f_k,\,\,f_k\in \mathcal{H}_k,\,\,N=1,\,2,\,\ldots\right\}$$
be the space of harmonic polynomials with null average on $\mathbb{S}^{n-1}$. For a fixed $f\in \mathcal{E}_0$, let $H$ be the solution in $\mathbb{B}^n$ of the Neumann problem with boundary data $f$, normalized so that $H(0)=0$. This will be expressed by the equation
$$ \left(\frac{\partial}{\partial \nu}\right)^{-1}f=H|_{\mathbb{S}^{n-1}},$$
where $\nu$ is the outward pointing normal vector to $\mathbb{S}^{n-1}$. One easily extends $(\partial/\partial \nu)^{-1}$ to $L^2_0(\mathbb{S}^{n-1})$ by the following formula: if $f=\sum_{k\geq 1} f_k$ is the decomposition of $f$ into spherical harmonics, then $(\partial/\partial\nu)^{-1}f=\sum_{k\geq 1}f_k/k$. Then $R^b$, \emph{the Riesz transform of ball type}, is defined by the identity
$$ R^b=\nabla_{\mathbb{S}^{n-1}}\circ \left(\frac{\partial}{\partial \nu}\right)^{-1}.$$
We will also work with the directional Riesz transform (of ball type), given by 
$$ Q^b_{\ell m}=\mathcal{T}_{\ell m}\circ \left(\frac{\partial }{\partial \nu}\right)^{-1},$$
as well as the auxiliary Riesz transform of ball type, defined by $Q^b=(Q^b_{\ell m})_{1\leq \ell<m\leq n}$. Applying \eqref{grad}, it is easy to check that $|R^b|=|Q^b|$ and thus it suffices to study the behavior of the operator $Q^b$.

Now we will describe the probabilistic representation of the above Riesz transforms. Let $B=(B^1,B^2,\ldots,B^n)$ be the standard Brownian motion in $\R^n$, starting from $0$, and let $\tau=\inf\{t\geq 0:B_t\notin\mathbb{B}^n\}$ be the first exit time of $B$ from the unit ball. Note that $B_\tau$ has the uniform distribution on $\mathbb{S}^{n-1}$. Let $A$ be a continuous function on the closed unit ball, with values in the class of $n\times n$ matrices. This function gives rise to the following martingale transformer: if $f\in C^\infty(\mathbb{S}^{n-1})$ and $F$ denotes its Poisson extension to $\mathbb{B}^n$, then
$$ A*F=\left(\int_0^{\tau\wedge t}A(B_s)\nabla_{\R^n}F(B_s)\cdot \mbox{d}B_s\right)_{t\geq 0}.$$
We define the $A$-transform of $f$ by the conditional expectation
$$ T_Af(x)=\E\big[A*F|B_\tau=x\big], \qquad x\in \mathbb{S}^{n-1}.$$
The connection between the operators $T_A$ and directional Riesz transforms is explained in the following statement, see Arcozzi \cite{Ar}.

\begin{theorem}\label{repr_sphere}
For given $1\leq \ell<m\leq n$, a function $\varphi:[0,1]\to \R$ and $x\in \overline{\mathbb{B}}^n$, let $A_{\ell m}(x)$ be the matrix with entries
$$ A_{\ell m}^{ij}(x)=\begin{cases}
\varphi(|x|^2) & \mbox{if }i=\ell,\,j=m,\\
-\varphi(|x|^2) & \mbox{if }i=m,\,j=\ell,\\
0 & \mbox{otherwise}.
\end{cases}$$

(i) If $\varphi\equiv 1$, then $ T_{A_{\ell m}}=Q_{\ell m}^b.$

(ii) Let $\varphi$ be defined by the formula
$$ \varphi(e^{-2t/(n-2)})=\frac{\int_0^t I_0(s)\mbox{d}s}{e^t-1},\qquad t\geq 0,$$
where $I_0(z)=\sum_{j=0}^\infty (z/2)^{2j}/(j!)^2$, $z\in\mathbb{C}$, is the modified Bessel function of order $0$. Then 
$T_{A_{\ell m}}=Q_{\ell m}^c.$ 
\end{theorem}

We are ready to establish the bounds for Riesz transforms. We start with the vectorial setting.

\begin{theorem}\label{sphere1} 
(i) For any $K>2/\pi$, any $f:\mathbb{S}^{n-1}\to \R$ with $\int_{\mathbb{S}^{n-1}}\Psi(|f|)<\infty$ and any Borel subset $E$ of $\mathbb{S}^{n-1}$ we have
\begin{equation}\label{vpsc1}
\int_E \left|R^{\mathbb{S}^{n-1}}f(x)\right|\mbox{d}x\leq 2K\int_{\mathbb{S}^{n-1}}\Psi(|f(x)|)\mbox{d}x+ \frac{|E|}{K-1},
\end{equation}
\begin{equation}\label{vpsb1}
\int_E |R^{b}f(x)|\mbox{d}x\leq 2(n-1)^{1/2}K\int_{\mathbb{S}^{n-1}}\Psi(|f(x)|)\mbox{d}x+ \frac{(n-1)^{1/2}|E|}{K-1}.
\end{equation}

(ii) For any $1<p<\infty$, any $f\in L^p(\mathbb{S}^{n-1})$ and any Borel subset $E$ of $\mathbb{S}^{n-1}$ we have
\begin{equation}\label{vpsc2}
\int_E \left|R^{\mathbb{S}^{n-1}}f(x)\right|\mbox{d}x\leq 2K_p||f||_{L^p(\mathbb{S}^{n-1})}|E|^{1-1/p},
\end{equation}
\begin{equation}\label{vpsb2}
\int_E |R^{b}f(x)|\mbox{d}x\leq 2(n-1)^{1/2}K_p||f||_{L^p(\mathbb{S}^{n-1})}|E|^{1-1/p}.
\end{equation}
\end{theorem}
\begin{proof}
We will only establish (i), the second part of the Theorem is shown in a similar manner. We start with \eqref{vpsb1} in which the reasoning is slightly easier. Consider the sequence $\mathcal{A}=(A_{\ell m})_{1\leq \ell<m\leq n}$, where $A_{\ell m}$ are as in Theorem \ref{repr_sphere} (i). For a  function $f\in C_0^\infty(\mathbb{S}^{n-1})$, let $F$ denote its Poisson extension to $\overline{\mathbb{B}^n}$. Introduce the martingales
$$ \xi_t=\big(F(B_{\tau\wedge t})\big)_{t\geq 0}=\left(\int_0^{\tau\wedge t}A_{\ell m}(B_s)\nabla_{\R^n}F(B_s)\cdot \mbox{d}B_s\right)_{t\geq 0}.$$
$$ \zeta_t=\mathcal{A}*F=\left(\left(\int_0^{\tau\wedge t}A_{\ell m}(B_s)\nabla_{\R^n}F(B_s)\cdot \mbox{d}B_s\right)_{1\leq \ell<m\leq n}\right)_{t\geq 0}$$
taking values in $\R^n$ and $\R^{n(n-1)/2}$, respectively. Since 
$$\sum_{1\leq \ell<m\leq n}|A_{\ell m} v|^2=(n-1)|v|^2\qquad \mbox{ for all $v\in \R^n$},$$ we conclude that $(n-1)^{-1/2}\zeta$ is differentially subordinate to $\xi$. Therefore, the inequality \eqref{intin} yields
$$ \E \max\left\{(n-1)^{-1/2}|\zeta_\infty|-\frac{1}{2(K-1)},0\right\}\leq K\E \Psi( |\xi_\infty|).$$
An application of the conditional version of Jensen's inequality gives
\begin{align*}
 &\int_{\mathbb{S}^{n-1}}\max\left\{(n-1)^{-1/2}|Q^bf(x)|-\frac{1}{2(K-1)},0\right\}\mbox{d}x\\
 &\qquad =\int_{\mathbb{S}^{n-1}}\max\left\{(n-1)^{-1/2}|T_{\mathcal{A}}f(x)|-\frac{1}{2(K-1)},0\right\}\mbox{d}x\\
 &\qquad = \E\max\left\{(n-1)^{-1/2}|T_{\mathcal{A}}f(B_\tau)|-\frac{1}{2(K-1)},0\right\}\mbox{d}x\\
 &\qquad \leq \E\max\left\{(n-1)^{-1/2}|A*F(B_\tau)|-\frac{1}{2(K-1)},0\right\}\mbox{d}x\\
 &\qquad \leq K\E\Psi(|\xi_\infty|)\\
 &\qquad =K\int_{\mathbb{S}^{n-1}}\Psi(|f(x)|)\mbox{d}x.
 \end{align*}
Now, for a given $E\subset \mathbb{S}^{n-1}$, we consider its decomposition into
\begin{align*}
 E^-&=E\cap\{|Q^bf(x)|<(n-1)^{1/2}/(K-1)\},\\
 E^+&=E\cap\{|Q^bf(x)|\geq (n-1)^{1/2}/(K-1)\},
\end{align*}
and, as previously, consider the integrals of $Q^b$ over $E^-$ and $E^+$ separately. This yields \eqref{vpsb1}.

We turn to the estimate \eqref{vpsc1}. The above reasoning would lead to a version with an additional factor $(n-1)^{1/2}$; to remove it, we will make use of a transference-type argument which enables to deduce the bound from the corresponding estimate on $SO(n)$. Imbedding this Lie group into $\R^{n^2}$ induces a bi-invariant Riemannian metric on $SO(n)$. This metric can be normalized so that the collection $\{X_{\ell m}=[r^{j,k}_{\ell m}]_{1\leq j,k\leq n}:1\leq \ell<m\leq n\}$, with
$$ r_{\ell m}^{j,k}=\begin{cases}
1 &\mbox{if }j=m,\,k=\ell,\\
-1 & \mbox{if }j=\ell,\,k=m,\\
0 & \mbox{otherwise,}
\end{cases}$$
forms an orthonormal basis in $\mathfrak{so}(n)$. Let $m_{SO(n)}$ be the normalized Haar measure on $SO(n)$. We identify $\mathbb{S}^{n-1}$ with $SO(n)/SO(n-1)$, where $SO(n-1)$ is the stabilizer of the northern pole $e_n=(0,0,\ldots,0,1)\in\mathbb{S}^{n-1}$. Let $\Pi:SO(n)\to\mathbb{S}^{n-1}$ be the projection given by $\Pi(a)=ae_n$, the image of $e_n$ under the rotation $a$. As shown by Arcozzi \cite{Ar}, the operators $Q^c_{\ell m}$ and $R_{\ell m}^{SO(n)}$ are related to each other by the formula
$$ Q_{\ell m}^cf(\Pi(a))=-R_{\ell m}^{SO(n)}(f\circ \Pi\circ \rho)(\rho(a)),$$
where $\rho(a)=a^{-1}$. Thus the estimate \eqref{vpsc1} follows from \eqref{vpg1}. To see this, note that for any $f:\mathbb{S}^{n-1}\to \R$ we have
$$ \int_{SO(n)} f\circ \Pi \mbox{d}m_{SO(n)}=\int_{\mathbb{S}^{n-1}}f \mbox{d}x.$$
Consequently, for any $E\subset \mathbb{S}^{n-1}$,
\begin{align*}
 \int_E |Q^cf(x)|\mbox{d}x&=\int_{\Pi^{-1}(E)}|Q^cf(\Pi(a))|\mbox{d}m_{SO(n)}(a)\\
&=\int_{\Pi^{-1}(E)}|R^{SO(n)}(f\circ \Pi\circ \rho)(\rho(a))|\mbox{d}m_{SO(n)}(a)\\
&\leq 2K\int_{SO(n)}\Psi(|f\circ \Pi\circ \rho|)\mbox{d}m_{SO(n)}+\frac{m_{SO(n)}(\Pi^{-1}(E))}{K-1}\\
&=2K\int_{\mathbb{S}^{n-1}}\Psi(|f(x)|)\mbox{d}x+\frac{|E|}{K-1}.
\end{align*}
The proof is complete.
\end{proof}

Finally, let us prove the logarithmic and weak-type bounds for directional Riesz transforms.

\begin{theorem}\label{sphere2}
Let $1\leq \ell<m\leq n$ be fixed and let $Q\in \{Q_{\ell m}^c,Q_{\ell m}^b\}$. 

(i) For $K>2/\pi$, any $f:\mathbb{S}^{n-1}\to \R$ with $\int_{\mathbb{S}^{n-1}}\Psi(|f|)<\infty$, and any Borel subset $E$ of $\mathbb{S}^{n-1}$ we have
\begin{equation}\label{rs3}
\int_E |Qf(x)|\mbox{d}x\leq K\int_{\mathbb{S}^{n-1}}\Psi(|f(x)|)\mbox{d}x+L(K)\cdot|E|.
\end{equation}

(ii) For all $1<p<\infty$, $f\in L^p(\mathbb{S}^{n-1})$ and any Borel subset $E$ of $\mathbb{S}^{n-1}$ we  have
\begin{equation}\label{rs4}
\int_E|Qf(x)|\mbox{d}x\leq C_p||f||_{L^p(\mathbb{S}^{n-1})}|E|^{1-1/p}.
\end{equation}
\end{theorem}
\begin{proof}
To show \eqref{rs3}, we establish first the following dual estimate: if $f:\mathbb{S}^{n-1}\to [-1,1]$, then 
\begin{equation}\label{rs1}
 \int_{\mathbb{S}^{n-1}} \Phi\left(|Qf(x)|/K\right)\mbox{d}\chi(x)\leq \frac{L(K)||f||_{L^1(\mathbb{S}^{n-1})}}{K}.
\end{equation}
The random variable $B_\tau$ has the uniform distribution on $\mathbb{S}^{n-1}$, so in view of Jensen's inequality,
$$ \int_{\mathbb{S}^{n-1}} \Phi\left(|Qf(x)|/K\right)\mbox{d}x=\E \Phi(|T_{A_{\ell m}}f(B_\tau)|)\leq \E \Phi(|(A_{\ell m}*F)_\infty|).$$
However, we have $\langle A_{\ell m} v,v\rangle=0$ and $||A_{\ell m}v||\leq ||v||$ for any $v\in \R^n$, since $0\leq \varphi\leq 1$. The latter bound is obvious in the ball type, in the cylindrical case one has to write down the expansion of $I_0$ to get that $0<I_0(s)\leq e^s$ and $I_0(0)=1$. Thus, $A_{\ell m}*F$ is orthogonal and differentially subordinate to  the martingale $F(B)=(\int_0^t \nabla_{\R^n}F(B_s)\cdot \mbox{d}B_s)_{t\geq 0}$ and hence, by Theorem \ref{martingale_thm2},
$$ \E \Phi(|(A_{\ell m}*F)_\infty|)\leq \frac{L(K)||F(B_\tau)||_1}{K}=\frac{L(K)||f||_{L^1(\mathbb{S}^{n-1})}}{K},$$
which is \eqref{rs1}. To deduce \eqref{rs3}, note that 
\begin{equation}\label{duality}
\int_{\mathbb{S}^{n-1}} Qf(x)g(x)\mbox{d}x=-\int_{\mathbb{S}^{n-1}} f(x)Qg(x)\mbox{d}x
\end{equation}
for all $f,\,g\in L^2(\mathbb{S}^{n-1})$. Let us briefly prove it. In the cylindrical case, if $f$, $g\in\mathcal{H}_k$, then
\begin{align*}
\int_{\mathbb{S}^{n-1}} Qf(x)g(x)\mbox{d}x&=
 \int_{\mathbb{S}^{n-1}} \mathcal{T}_{\ell m}(\Delta_{\mathbb{S}^{n-1}})^{-1/2}f(x)g(x)\mbox{d}x\\\
 &=\frac{1}{\sqrt{k(n+k-2)}}\int_{\mathbb{S}^{n-1}} T_{\ell m}f(x) g(x)\mbox{d}x\\
 &=-\frac{1}{\sqrt{k(n+k-2)}}\int_{\mathbb{S}^{n-1}} f(x) T_{\ell m}g(x)\mbox{d}x\\
 &=-\int_{\mathbb{S}^{n-1}} f(x)Qg(x)\mbox{d}x.
 \end{align*}
On the other hand, if $f$, $g$ belong to two different classes $\mathcal{H}_j$, $\mathcal{H}_k$ and we extend them to homogeneous polynomials on the whole $\R^n$, then, using Green's formula, we get
\begin{align*}
k\int_{\mathbb{S}^{n-1}}\mathcal{T}_{\ell m}f(x)g(x)\mbox{d}x
&=\int_{\mathbb{S}^{n-1}}\mathcal{T}_{\ell m}\frac{\partial f}{\partial \nu}(x)g(x)\mbox{d}x\\
&=\int_{\mathbb{S}^{n-1}}\frac{\partial}{\partial \nu}\mathcal{T}_{\ell m}f(x)g(x)\mbox{d}x\\
&=\int_{\mathbb{S}^{n-1}}\mathcal{T}_{\ell m}f(x)\frac{\partial}{\partial \nu}g(x)\mbox{d}x\\
&=\ell \int_{\mathbb{S}^{n-1}}\mathcal{T}_{\ell m}f(x)g(x)\mbox{d}x,
\end{align*}
and hence $\int_{\mathbb{S}^{n-1}}Qf(x)g(x)\mbox{d}x=-\int_{\mathbb{S}^{n-1}}f(x)Qg(x)\mbox{d}x=0$. Thus, \eqref{duality} follows by expanding $f$ and $g$ in the series of spherical harmonics. If $Q$ is of ball type, then \eqref{duality} is proved with the use of similar arguments.
 Now, pick an arbitrary Borel subset $E$ of $\mathbb{S}^{n-1}$ and put $g(x)=\chi_E(x)\cdot Qf(x)/|Qf(x)|$ for $x\in \mathbb{S}^{n-1}$ (with the convention $g=0$ if $Qf=0$). Using \eqref{duality} and then \eqref{rs1}, we obtain
\begin{align*}
 \int_E |Qf(x)|\mbox{d}x&=\int_{\mathbb{S}^{n-1}}Qf(x)g(x)\mbox{d}x\\
 &=-\int_{\mathbb{S}^{n-1}}f(x)Qg(x)\mbox{d}x\\
 &\leq  K\int_{\mathbb{S}^{n-1}}\Psi(|f(x)|)\,\mbox{d}x+K\int_{\mathbb{S}^{n-1}}\Phi(|R_jg(x)|/K)\,\mbox{d}x\\
&\leq  K\int_{\mathbb{S}^{n-1}}\Psi(|f(x)|)\,\mbox{d}x+L(K)||g||_{L^1(\mathbb{S}^{n-1})}\\
&\leq  K\int_{\mathbb{S}^{n-1}}\Psi(|f(x)|)\,\mbox{d}x+L(K)|E|
\end{align*} 
and \eqref{rs3} follows. The proof of \eqref{rs4} is similar and exploits the dual bound
$$ ||Qf||_{L^q(\mathbb{S}^{n-1})}\leq C_p||f||_{L^1(\mathbb{S}^{n-1})}^{1/q}||f||_{L^\infty(\mathbb{S}^{n-1})}^{1/p}.$$
The details are left to the reader.
\end{proof}

\subsection{Logarithmic and weak-type inequalities for Riesz transforms on Gauss space}\label{GaussRiesz}
Throughout this section, $\mathbb{S}_n$ denotes the $n-1$-dimensional sphere of radius $\sqrt{n}$, equipped with its natural Riemannian metric and $SO(n)$ invariant measure $\mu_n$ satisfying $\mu_n(\mathbb{S}_n)=1$. With $\mathcal{T}_{\ell m}$ as in \eqref{defT} and a smooth function $f:\mathbb{S}_n\to \R$, we have
\begin{equation}\label{diffS} 
\Delta_{\mathbb{S}_n}f=\frac{1}{n}\sum_{1\leq \ell<m\leq n} \mathcal{T}_{\ell m}\mathcal{T}_{\ell m}f,\qquad  |\nabla_{\mathbb{S}_n} f|^2=\frac{1}{n}\sum_{1\leq \ell<m\leq n} |\mathcal{T}_{\ell m}|^2.
\end{equation}

 A well-known and frequently used fact (cf. \cite{M}) is that many geometric objects on $\mathbb{S}_n$ pass in the limit to the corresponding objects on Gauss space; this is often referred to as \emph{Poincar\'e's limit} or \emph{Poincar\'e's observation}, though the argument can be tracked back to the work of Mehler \cite{Meh}. The purpose of this section is to present another illustration for this phenomenon. Namely, we will show how the estimates for cylindrical Riesz transforms lead to analogous bounds for the Riesz transforms associated with the Ornstein-Uhlenbeck semigroup, fundamental tools in the Malliavin calculus on the Wiener space \cite{Mal}. 

We start with the necessary notation. Let $d$ be a fixed positive integer and suppose that $\gamma_d$ is the standard Gaussian measure on $\R^d$, i.e.,
$$ \mbox{d}\gamma_d(x)=(2\pi)^{-d/2}\exp(-|x|^2/2)\mbox{d}x,\qquad x\in \R^d.$$
Let $\nabla_{\R^d}^*$ be the formal adjoint of the gradient $\nabla_{\R^d}$ in $L^2(\R^d,\gamma_d)$. Then
$$ L=\nabla_{\R^d}^*\nabla_{\R^d}=\Delta_{\R^d}-x\cdot \nabla_{\R^d}$$
is a negative operator which generates  Ornstein-Uhlenbeck semigroup in $d$ dimensions. The Riesz transform associated with $L$ is defined by
$$ R^{L}=\nabla_{\R^d}\circ (-L)^{-1/2}.$$

Next, fix $n\geq d$ and define the ``projection'' $\Pi_n:\mathbb{S}_n\to \R^d$ by $\Pi_n(x,y)=x$, where $x\in \R^d$, $y\in \R^{n-d}$ and $(x,y)\in \mathbb{S}_n$. For an arbitrary function $f:\mathbb{R}^d\to \R$, we will write $f_n=f\circ \Pi_n$. Poincar\'e's observation \cite{Meh} amounts to saying that for any measurable subset $E$ of $\R^d$ we have
$$ \lim_{n\to \infty} \int_{\mathbb{S}_n}(\chi_E)_n\mbox{d}\mu_n =\int_{\R^d} \chi_E \mbox{d}\gamma_d.$$
This can be pushed further: if a function $f:\R^n\to \R$ has polynomial growth, then
\begin{equation}\label{Mehler}
 \lim_{n\to\infty}\int_{\mathbb{S}_n} f_n\mbox{d}\mu_n =\int_{\R^d}f\mbox{d}\gamma_d.
\end{equation}
As a consequence, we obtain that for such $f$,
\begin{equation}\label{st}
\begin{split}
\lim_{n\to \infty}||\nabla_{\mathbb{S}_n}f_n||_{L^p(\mathbb{S}_n)}&=||\nabla_{\R^d} f||_{L^p(\R^d,\gamma_d)},\\
\lim_{n\to\infty}||\Delta_{\mathbb{S}_n}f_n||_{L^p(\mathbb{S}_n)}&=||Lf||_{L^p(\R^d,\gamma_d)},
\end{split}
\end{equation}
where $L$ is the generator of Ornstein-Uhlenbeck semigroup introduced above. These equalities follow immediately from \eqref{Mehler} and the identities (cf. \cite{M})
$$ |\nabla_{\mathbb{S}_n}f_n|^2=\left[\sum_{j=1}^d (\partial_jf)^2-\frac{1}{n}\left(\sum_{j=1}^d x_j\partial_jf\right)^2\right]_n$$
and
$$ \Delta_{\mathbb{S}_n}f_n=\left[\sum_{j=1}^d \partial_{jj}^2f-\frac{n-1}{n}\sum_{j=1}^d x_j\partial_jf-\frac{1}{n}\sum_{j=1}^d\sum_{k=1}^d x_jx_k\partial_{jk}^2f\right]_n.$$

Let $\mathfrak{H}_k^d$ denote the space of generalized Hermite polynomials of degree $k$ on $\R^d$, i.e., the space of those polynomials $P:\R^d\to \R$, which satisfy $\operatorname*{deg}P\leq k$ and $LP+kP=0$. This class is closely related to the space of spherical harmonics on $\R^n$. To describe the connection, pick $P\in \mathfrak{H}_k^d$, a number $n>d$, and consider the decomposition
\begin{equation}\label{decomposition}
 P_n=\sum_{j\leq k} Q_j^{n,d}(P).
\end{equation}
Here $Q_j^{n,d}$ is the $L^2(\mathbb{S}_n)$-orthogonal projection of $P$ onto $\mathcal{H}_j(\R^n)$, the space of spherical harmonics of degree $j$, extended to a homogeneous polynomial on $\R^n$. It turns out that among all the summands $Q_j^{n,d}(P)$, the term $Q_n^{n,d}(P)$ has an overwhelming size. We will need the following statement: see Lemma 6.1 and Lemma 6.2 in Arcozzi \cite{Ar}.

\begin{lemma}\label{hol}
Let $P\in \mathfrak{H}_k^d$ and consider its decomposition \eqref{decomposition}. Then for any $1\leq p<\infty$, 
\begin{equation}\label{hol1}
 \lim_{n\to \infty}||Q_j^{n,d}(P)||_{L^p(\mathbb{S}_n)}=0\qquad \mbox{if }j<k,
 \end{equation}
and
\begin{equation}\label{hol2}
 \limsup_{n\to\infty}||Q_k^{n,d}(P)||_{L^p(\mathbb{S}_n)}\leq K_{p,k,d}||P||_{L^2(\R^d,\gamma_d)},
\end{equation}
where the constant $K_{p,k,d}$ depends only on the parameters indicated.
\end{lemma}

We turn to the main result of this section.

\begin{theorem}\label{gauss1}
(i) For any $K>2/\pi$, any $f:\R^d\to \R$ with $\int_{\R^d}\Psi(|f|)\mbox{d}\gamma_d<\infty$ and any Borel subset $E$ of $\R^d$ we have
\begin{equation}\label{vgs1}
\int_E |R^{L}f(x)|\mbox{d}\gamma_d(x)\leq 2K\int_{\R^d}\Psi(|f(x)|)\mbox{d}\gamma_d(x)+ \frac{\gamma_d(E)}{K-1}.
\end{equation}

(ii) For any $1<p<\infty$, any $f\in L^p(\R^n,\gamma_d)$ and any Borel subset $E$ of $\R^n$ we have
\begin{equation}\label{vgs2}
\int_E |R^{L}f(x)|\mbox{d}\gamma_d(x)\leq 2K_p||f||_{L^p(\R^n,\gamma_d)}\gamma_d(E)^{1-1/p}.
\end{equation}
\end{theorem}
\begin{proof}
Suppose that $P=P^{(1)}+P^{(2)}+\ldots+P^{(N)}$, where $P^{(k)}\in\mathfrak{H}_k^d$, $k=1,\,2,\,\ldots,\,N$, and let $1\leq p<\infty$ be a fixed number. Let us exploit the decomposition \eqref{decomposition} for $P^{(k)}$ to get
\begin{equation}\label{each}
 (-\Delta_{\mathbb{S}_n})^{1/2}P_n^{(k)}=(-\Delta_{\mathbb{S}_n})^{1/2}\sum_{j\leq k} Q_j^{n,d}(P^{(k)}).
\end{equation}
Since the restriction of $Q_k^{n,d}$ is a spherical harmonic of degree $k$, we may write
\begin{align*}
 (-\Delta_{\mathbb{S}_n})^{1/2}Q_k^{n,d}(P^{(k)})&=\sqrt{k(n-2+k)/n}Q_k^{n,d}(P^{k)})\\
 &=\sqrt{k}Q_k^{n,d}(P^{(k)}+\left(\sqrt{\frac{n-2+k}{n}}-1\right)\sqrt{k}Q_k^{n,d}(P^{(k)}).
 \end{align*}
and, similarly for $j<k$,
$$  (-\Delta_{\mathbb{S}_n})^{1/2}Q_j^{n,d}(P^{(k)})=\sqrt{j(n-2+j)/n}Q_j^{n,d}(P^{(k)}).$$
Plug the above expressions into \eqref{each} and apply triangle inequality to obtain
\begin{equation}\label{ify}
\begin{split}
 \left|\left| (-\Delta_{\mathbb{S}_n})^{1/2}P_n\right|\right|_{L^p(\mathbb{S}_n)}&=\left|\left| \sum_{k=1}^N(-\Delta_{\mathbb{S}_n})^{1/2}P_n^{(k)}\right|\right|_{L^p(\mathbb{S}_n)}\\
&=
\left|\left|\sum_{k=1}^N \sqrt{k}Q_k^{n,d}(P^{(k)})\right|\right|_{L^p(\mathbb{S}_n)}+\eta_n, 
\end{split}
\end{equation}
where the error term $\eta_n$ is bounded from above by
\begin{align*} \sum_{k=1}^N &\left(\sqrt{\frac{n-2+k}{n}}-1\right)\sqrt{k}||Q_k^{n,d}(P^{(k)})||_{L^p(\mathbb{S}^n)}\\
&\qquad +\sum_{1\leq j<k\leq N}\sqrt{\frac{j(n-2+j)}{n}}
||Q_j^{n,d}(P^{(k)})||_{L^p(\mathbb{S}^n)}.
\end{align*}
Note that both above sums tend to $0$ as $n\to \infty$, in view of Lemma \ref{hol}. To see the convergence of the first sum, simply use \eqref{hol2} and the fact that $\sqrt{(n-2+k)/n}\to 1$ as $n\to \infty$; to analyze the second sum, apply \eqref{hol1}.

We come back to \eqref{ify}. Applying \eqref{decomposition} again, we write
$$\left|\left|\sum_{k=1}^N \sqrt{k}Q_k^{n,d}(P^{(k)})\right|\right|_{L^p(\mathbb{S}_n)}= \left|\left|\sum_{k=1}^N \sqrt{k}P^{(k)}_n\right|\right|_{L^p(\mathbb{S}_n)}+\kappa_n,$$
where, by the triangle inequality,
$$ \kappa_n\leq \sum_{1\leq j<k\leq d}\sqrt{k}\left|\left|Q_j^{n,d}(P^{(k)})\right|\right|_{L^p(\mathbb{S}_n)}.$$
Finally, note that by Mehler's observation \eqref{Mehler},
\begin{align*}
\lim_{n\to\infty} \left|\left|\sum_{k=1}^N \sqrt{k}P^{(k)}_n\right|\right|_{L^p(\mathbb{S}_n)}&=\lim_{n\to\infty}\left|\left|\left(\sum_{k=1}^N \sqrt{k}P^{(k)}\right)_n\right|\right|_{L^p(\mathbb{S}_n)}\\
 &= \left|\left|\sum_{k=1}^N \sqrt{k}P^{(k)}\right|\right|_{L^p(\R^d,\gamma_d)}\\
 &=\left|\left|(-L)^{1/2}P\right|\right|_{L^p(\R^d,\gamma_d)}.
 \end{align*}
Combining all the above facts, we obtain the convergence
\begin{equation}\label{convergence}
\lim_{n\to\infty} \left|\left| (-\Delta_{\mathbb{S}_n})^{1/2}P_n\right|\right|_{L^p(\mathbb{S}_n)}
=\left|\left|(-L)^{1/2}P\right|\right|_{L^p(\R^d,\gamma_d)}.
\end{equation}
A similar argumentation (based on the bound $|\Psi(t)-\Psi(s)|\leq |t^2-s^2|$) shows that
\begin{equation}\label{convergence2}
\lim_{n\to\infty} \int_{\mathbb{S}_n}\Psi\Big(\big|(-\Delta_{\mathbb{S}_n})^{1/2}P_n\big|\Big)\mbox{d}\mu_n
=\int_{\R^d}\Psi\Big(\big|(-L)^{1/2}P\big|\Big)\mbox{d}\gamma_d.
\end{equation}

We are ready to establish the assertion of the theorem. Let us use the inequality \eqref{vpsc1} with the set $\Pi^{-1}(E)/\sqrt{n}\subseteq \mathbb{S}^{n-1}$ and the function $f=\frac{1}{\sqrt{n}}(-\Delta_{\mathbb{S}^{n-1}})^{1/2}(P_n\circ \rho)$, where $\rho:\mathbb{S}^{n-1}\to \mathbb{S}_n$ is given by $\rho(x)=x\sqrt{n}$. Using \eqref{diffS}, we easily compute that $|Q^cf|=|R^cf|=|(\nabla_{\mathbb{S}_n}P_n)\circ\rho|$ and $f=\big((-\Delta_{\mathbb{S}_n})^{1/2}P_n\big)\circ \rho$, so we get
\begin{align*}
 &\int_{\Pi^{-1}(E)/\sqrt{n}} |(\nabla_{\mathbb{S}_n}P_n)\circ\rho(x)|\mbox{d}x\\
&\qquad \qquad \leq 2K\int_{\mathbb{S}^{n-1}} \Psi\big(\big|\big((-\Delta_{\mathbb{S}_n})^{1/2}P_n\big)\circ \rho\big|\big)\mbox{d}x+\frac{|\Pi^{-1}(E)/\sqrt{n}|}{K-1}.
\end{align*}
Hence, substituting $z=x\sqrt{n}$ in the two integrals, we obtain
\begin{align*}
 \int_{\mathbb{S}_n} (\chi_E)_n |\nabla_{\mathbb{S}_n}P_n|\mbox{d}\mu(x)&\leq 2K\int_{\mathbb{S}_n}\Psi\Big(\big|(-\Delta_{\mathbb{S}_n})^{1/2}P_n\big|\Big)\mbox{d}\mu_n+\frac{\int_{\mathbb{S}^n}(\chi_E)_n\mbox{d}\mu_n}{K-1}.
\end{align*}
Letting $n\to \infty$ yields
$$ \int_E |\nabla_{\R^n}P|\mbox{d}\gamma_d\leq 2K\int_{\R^d}\Psi\Big(\big|(-L)^{1/2}P\big|\Big)\mbox{d}\gamma_d+\frac{\gamma_d(E)}{K-1}.$$
Putting $f=(-L)^{-1/2}P$, we obtain \eqref{vgs1} for finite linear combinations of Hermite polynomials. By density, the estimate extends to all $f$ satisfying $\int_{\R^d}\Psi(|f|)\mbox{d}\gamma_d<\infty$.

The proof of \eqref{vgs2} goes along the same lines.
\end{proof}

\def\wrong{
We turn to the statement for directional Riesz transforms.

\begin{theorem}
Let $1\leq j\leq d$ be fixed and let $R_j$ denote the directional Riesz transform associated with the Ornstein-Uhlenbeck operator.

(i) For $K>2/\pi$, $f:\R^n\to \R$ with $\int_{\R^n}\Psi(|f|)\mbox{d}\gamma_d<\infty$ and any Borel subset $E$ of $\R^n$ we have
\begin{equation}\label{rg3}
\int_E |R_jf(x)|\mbox{d}\gamma_d(x)\leq K\int_{\R^n}\Psi(|f(x)|)\mbox{d}\gamma_d+L(K)\cdot\gamma_d(E).
\end{equation}

(ii) For all $1<p<\infty$, any $f\in L^p(\R^n,\gamma_d)$ and any Borel subset $E$ of $\R^n$ we  have
\begin{equation}\label{rg4}
\int_E|R_jf(x)|\mbox{d}\gamma_d(x)\leq C_p||f||_{L^p(\R^n,\gamma_d)}\gamma_d(E)^{1-1/p}.
\end{equation}
\end{theorem}
\begin{proof} The argument is similar to that used in the previous proof. 
Apply $Q^c_{j,d+1}$ to the function $f=\frac{1}{n}(-\Delta_{\mathbb{S}^{n-1}})^{1/2}(P_n\circ \rho)$ (here, as previously, $\rho:\mathbb{S}^{n-1}\to\mathbb{S}_n$ is given by  $\rho(x)=x\sqrt{n}$). We get
$$ Q^c_{j,d+1}f=x_j\partial_{d+1}P_n(x)-x_{d+1}\partial_jP_n(x)=-x_{d+1}(\partial_j P)_n(x).$$
Therefore, the inequality \eqref{rs3} gives
$$ \int_{\mathbb{S}_n}(\chi_{A})_n|x_{d+1}|\;|(\partial_jP)_n|\mbox{d}\mu_n\leq K\int_{\mathbb{S}_n}\Psi\Big(\big|(-\Delta_{\mathbb{S}_n})^{1/2}P_n\big|\Big)\mbox{d}\mu_n+L(K)\mu_n(A_n)$$
and letting $n\to\infty$ implies, in view of \eqref{Mehler} and \eqref{convergence2},
$$ \int_{\R^{d+1}} \chi_A|x_{d+1}|\partial_j P|\mbox{d}\gamma_{d+1}\leq K\int_{\R^n}\Psi\Big(\big|(-L)^{1/2}P\big|\Big)\mbox{d}\gamma_d+L(K)\cdot \gamma_d(A).$$
Substituting $f=(-L)^{1/2}P$ and using $\int_\R |x|\mbox{d}\gamma_1(x)=\sqrt{\frac{\pi}{2}}$, we obtain the estimate
$$ \int_A |R_jf|\mbox{d}\gamma_d\leq \sqrt{\frac{\pi}{2}}K\int_{\R^n}\Psi\Big(\big|f\big|\Big)\mbox{d}\gamma_d+\sqrt{\frac{\pi}{2}}L(K)\cdot \gamma_d(A),$$
and the general case follows easily by approximation. To show the weak-type bounds, we argue similarly: an application of \eqref{rs4} gives
$$ \int_{\mathbb{S}_n}(\chi_{A})_n|x_{d+1}|\;|(\partial_jP)_n|\mbox{d}\mu_n\leq C_p\mu_n(A_n)^{1/q}|| (-\Delta_{\mathbb{S}_n})^{1/2}P_n||_{L^p(\mathbb{S}_n)}.$$
If we let $n\to \infty$ and apply Mehler's obsevation \eqref{Mehler}, we get
$$ \int_{\R^{d+1}} \chi_A|x_{d+1}|\partial_j P\mbox{d}\gamma_{d+1}\leq C_p\gamma_d(A)^{1/q}||(-L)^{1/2}P||_{L^p(\mathbb{S}_n)},$$
or, we substitute $f=(-L)^{1/2}P$, 
$$ ||R_jf||_{L^{p,\infty}(\R^n,\gamma_d)}\leq \sqrt{\frac{2}{\pi}}C_p||f||_{L^p(\R^n,\gamma_d)}.$$
By density, this estimate holds true for all $f\in L^p(\R^n,\gamma_d)$.
\end{proof}
}

%%%%%%%%%%%%%%%%%%%%%%%%%%%%
%%%%%%%%%%%%%%%%%%%%%%%%%%%
\subsection*{Acknowledgements}  Rodrigo Ba\~nuelos gratefully acknowledges the many useful conversation with Fabrice Baudoin on topics related to this paper.

\end{document}